\documentclass[12pt,a4paper,final]{amsart}

\usepackage{float}
\usepackage{tikz}
\usepackage{amsmath, amsthm, verbatim, amsfonts, amssymb, xcolor}
\usepackage{graphics, xspace, enumerate}
\usepackage[a4paper,margin=1in]{geometry}
\usepackage{marginnote}

\usepackage[bb=boondox]{mathalfa}

\usepackage{graphicx}
\usepackage[a4paper,colorlinks=true,citecolor=blue,urlcolor=blue,linkcolor=blue,bookmarksopen=true,unicode=true,pdffitwindow=true]{hyperref}

\hypersetup{pdfauthor={}}
\hypersetup{pdftitle={Markov chain approximation of pure jump processes}}

\theoremstyle{plain}
\newtheorem{theorem}{Theorem}[section]
\newtheorem{corollary}[theorem]{Corollary}
\newtheorem{lemma}[theorem]{Lemma}
\newtheorem{proposition}[theorem]{Proposition}
\theoremstyle{definition}
\newtheorem{remark}[theorem]{Remark}
\newtheorem{example}[theorem]{Example}
\newtheorem{definition}[theorem]{Definition}
\newtheorem*{notation}{Notation}
\newtheorem*{ack}{Acknowledgement}

\newcommand {\I} {\mathbb{1}}
\newcommand {\real} {\ensuremath{\mathbb{R}}}
\newcommand {\rd} {\ensuremath{{\real^d}}}
\newcommand {\ZZ} {\ensuremath{\mathbb{Z}}}
\newcommand {\nat} {\ensuremath{\mathbb{N}}}

\newcommand{\process}[1]{\{#1_t\}_{t\geq0}}

\newcommand{\Cclip}{C_c^{\mathrm{Lip}}}

\newcommand{\et}{\quad\text{and}\quad}
\newcommand{\downto}{\downarrow}
\newcommand\supp{\mathop{\operatorname{supp}}}

\newcommand\diag{\mathrm{diag}}

\numberwithin{equation}{section}

\begin{document}
\allowdisplaybreaks[4]

\title{Markov Chain Approximation of Pure Jump Processes}

\author[A.\ Mimica]{Ante Mimica\,$(\dagger)$}
\address[Mimica]{Ante Mimica, *20-Jan-1981\,-- 
	$\dagger$\,9-Jun-2016 
	\\
\url{https://web.math.pmf.unizg.hr/~amimica/}}

\author[N.\ Sandri\'{c}]{Nikola Sandri\'{c}}
\address[Sandri\'{c}]{Faculty of Civil Engineering\\University of Zagreb\\10000 Zagreb\\Croatia}
\email{nsandric@grad.hr}

\author[R.\,L.~Schilling]{Ren\'e L.\ Schilling}
\address[Schilling]{Institut f\"{u}r Mathemati\-sche Stochastik\\Fachrichtung Mathematik\\TU Dresden\\01062 Dresden\\Germany}
\email{rene.schilling@tu-dresden.de}


\begin{abstract}
    In this paper we discuss weak convergence of continuous-time Markov chains to a non-symmetric pure jump process. We approach this problem using Dirichlet forms as well as semimartingales. As an application, we discuss how to approximate a given Markov process by Markov chains.
\end{abstract}

\maketitle

\section{Introduction}
Let $\mathbf{X}^{n}$, $n\in\nat$ be a sequence of continuous-time Markov chains where $\mathbf{X}^n$ takes values on the lattice $n^{-1}\mathbb{Z}^{d}$, and let $\mathbf{X}$ be a Markov process on $\real^{d}$. We are interested in the following two questions:
\begin{itemize}
    \item [(i)] Under which conditions does $\{\mathbf{X}^{n}\}_{n\in\nat}$ converge weakly to some (non-symmetric) Markov process?
    \item [(ii)] Can a given Markov process $\mathbf{X}$ be approximated (in the sense of weak convergence) by a sequence of Markov chains?
\end{itemize}
These questions have a long history. If $\mathbf{X}$ is a diffusion process determined by a generator in non-divergence form these problems have been studied in \cite{Stroock-Varadhan-Book-2006} using martingale problems. The key ingredient in this approach is that the domain of the corresponding generator is rich enough, i.e.\ containing the test functions $C_c^{\infty}(\real^{d})$. On the other hand, if the generator of $\mathbf{X}$ is given in divergence form, it is a delicate matter to find non-trivial functions in its domain. In order to overcome this problem, one resorts to an $L^{2}$-setting and the theory of Dirichlet forms; for example, \cite{Stroock-Zheng-1997} solve these problems for symmetric diffusion processes $\mathbf{X}$ using Dirichlet forms.
The main assumptions are certain uniform regularity conditions and the boundedness of the range of the conductances of the approximating Markov chains. These results are further extended in \cite{Bass-Kumagai-2008}, where the uniform regularity condition is relaxed and the conductances may have unbounded range. Very recently, \cite{Deuschel-Kumagai-2013} discusses these questions for a non-symmetric diffusion process $\mathbf{X}$. Let us also mention that the problem of approximation of a reflected Brownian motion on a bounded domain 
in $\real^{d}$ is studied in \cite{Burdzy-Chen-2008}.

As far as we know, the paper \cite{Husseini-Kassmann-2007} is among the first papers studying the approximation of a jump process $\mathbf{X}$. In this work the authors investigate convergence to and approximation of a symmetric jump process $\mathbf{X}$ whose jump kernel is  comparable to the jump kernel of a symmetric stable L\'evy process. These results have been extended in \cite{Bass-Kassmann-Kumagai-2010}, where the comparability assumption is imposed on the small jumps only, whereas the big jumps are controlled by a certain integrability condition. The case where $\mathbf{X}$ is a symmetric process which has both a continuous and a jump part is dealt with in \cite{Bass-Kumagai-Uemura-2010}.

 Let us point out that all of these approaches require some kind of ``stable-like" property (or control) of the jump kernel,
and the main step in the proofs is to obtain heat kernel estimates of the chains $\{\mathbf{X}^{n}\}_{n\in\nat}$. This is possible due to the uniform ellipticity assumption in the continuous case and the ``stable-like" assumption in the jump case. In general, this is very difficult to verify, and in many cases it is even impossible. Using a completely different approach, \cite{Chen-Kim-Kumagai-2013} study the convergence and approximation problems for pure jump processes $\mathbf{X}$ on a metric measure space satisfying the volume doubling condition. The proof of tightness is based on methods developed in \cite[Lemma 2.1]{Burdzy-Chen-2008} and only works if the approximating Markov chains $\mathbf{X}^n$ are symmetric. In order  to prove the convergence of the finite-dimensional distributions of $\{\mathbf{X}^{n}\}_{n\in\nat}$ to those of $\mathbf{X}$,
Mosco convergence of the corresponding symmetric Dirichlet forms is used. This type of convergence is equivalent to strong convergence of the corresponding semigroups. It was first obtained in \cite{Mosco-1994} in the case when all the forms are defined on the same Hilbert space, and then it was generalized in \cite{Kim-2006} (see also \cite{Chen-Kim-Kumagai-2013} and \cite{Kuwae-Shioya-2003}) to the case where the forms are defined on different spaces.

We are interested in the convergence and approximation problems for \emph{non-symmetric} pure jump processes. We will use two approaches: (i) via Dirichlet forms, and (ii) via semimartingale convergence results.  The first approach (Section~\ref{sec21}--\ref{sec23}) follows the roadmap laid out in \cite{Chen-Kim-Kumagai-2013}: To obtain tightness of $\{\mathbf{X}^{n}\}_{n\in\nat}$ we use semimartingale convergence results developed in \cite{Jacod-Shiryaev-2003}. More precisely, we first ensure that the processes $\mathbf{X}^{n}$, $n\in\nat$, are regular Markov chains (in particular, they are semimartingales), then we compute their semimartingale characteristics, and finally we provide conditions for the tightness of $\{\mathbf{X}^{n}\}_{n\in\nat}$ in terms of the corresponding conductances. This is based on a result from \cite{Jacod-Shiryaev-2003}  which states that a sequence of semimartingales is tight if the corresponding characteristics are $C$-tight (i.e. tight and  all accumulation points are processes with continuous paths). To get the convergence of the finite-dimensional distributions of $\{\mathbf{X}^{n}\}_{n\in\nat}$ in the non-symmetric case we can still use  Mosco convergence, but for non-symmetric Dirichlet  forms. Just as in the symmetric case, this type of convergence is equivalent to the strong convergence of the corresponding semigroups. It  was first obtained in \cite{Hino-1998} for forms defined on the same Hilbert space, and and then it was generalized in \cite{Tolle-Thesis-2006} to forms living on different spaces.

Our second approach (Section~\ref{sce31}), is based on a result from \cite{Jacod-Shiryaev-2003} which provides general conditions under which a sequence of semimartingales converges weakly to a semimartingale. If the processes $\mathbf{X}^{n}$, $n\in\nat$, are regular Markov chains and if the limiting process $\mathbf{X}$ is a so-called (pure jump) homogeneous diffusion with jumps, we obtain conditions (in terms of conductances and characteristics of $\mathbf{X}$) which imply the desired convergence.

As an application, we can now answer question (ii) and provide conditions for the approximation of a given Markov process, both in the Dirichlet form set-up (Section~\ref{sec24}) and the semimartingale setting (Section~\ref{sec32}).

\begin{notation}
    Most of our notation is standard or self-explanatory.
    Throughout this paper, we write $\ZZ^d_n := \frac 1n\ZZ^d = \{\frac 1n m : m\in\ZZ^d\}$ for the $d$-dimensional lattice with grid size $\frac 1n$.
    For $p\geq1$ we use $L^p_n$ as a shorthand for $L^p(\ZZ^d_n)$, the standard $L^p$-space on $\ZZ^d_n$. If $p=2$, the scalar product is given by $\langle f,g\rangle_{L_n^2} :=  \sum_{a\in\ZZ_n^d} f(a)g(a)$. We write $C_c^k(\rd)$ for the $k$-times continuously differentiable, compactly supported functions, and $C_c^{\mathrm{Lip}}(\real^{d})$ is the space of Lipschitz continuous functions with compact support. $B_r(x)$ is the open ball with radius $r>0$ and centre $x$, $\diag = \{(x,x) : x\in\rd\}$ denotes the diagonal in $\rd$.
    Finally, the sum $A+B$ of subsets $A,B\subseteq\real^d$ is defined as
    $A+B=\{a+b:a\in A, b\in B\}$.
\end{notation}

\section{Convergence of Markov chains using Dirichlet forms}\label{s2}
Our starting point is a sequence of continuous-time Markov chains $\process{X^{n}}$ with state space $\ZZ^d_n$ and infinitesimal generator
\begin{gather*}
    \mathcal{A}^{n}f(a)
    =\sum_{b\in\ZZ_n^{d}}(f(b)-f(a))C^{n}(a,b),\quad f\in\mathcal{D}_{\mathcal{A}^{n}},
\intertext{where the domain is given by}
    \mathcal{D}_{\mathcal{A}^{n}}
    :=\Big\{f:\ZZ_n^{d}\to\real : \sum_{b\in\ZZ_n^{d}}|f(b)|C^{n}(a,b)<\infty \text{\ for all\ } a\in\ZZ_n^{d}\Big\}.
\end{gather*}
A sufficient condition for the existence of $\mathbf{X}^n$ is that the kernel $C^{n}:\ZZ_n^{d}\times\ZZ_n^{d}\to[0,\infty)$ satisfies the following two properties
\begin{gather}
    \label{t1}\tag{\textbf{T1}} \forall n\in\nat,\:\forall a\in\ZZ_n^d \::\: C^{n}(a,a)=0;\\
    \label{t2}\tag{\textbf{T2}} \forall n\in\nat \::\: \sup_{a\in\ZZ_n^{d}}\sum_{b\in\ZZ_n^{d}}C^{n}(a,b)<\infty,
\end{gather}
see e.g.\ \cite{Norris-Book-1998}; in this case, the chain $\process{X^{n}}$ is regular, i.e.\ it has only finitely many jumps on finite time-intervals. If the chain is in state $a\in\ZZ^d$, it jumps to state $b\in\ZZ^d$ with probability $C^{n}(a,b)/\sum_{c\in\ZZ_n^{d}}C^{n}(a,c)$ after an exponential waiting time with parameter $\sum_{c\in\ZZ_n^{d}}C^{n}(a,c)$. Moreover, $\process{X^{n}}$ is conservative and defines a semimartingale. Observe that \eqref{t1} implies $L_n^{\infty}\cup L^1_n\cup L^2_n\subseteq\mathcal{D}_{\mathcal{A}^{n}}$. Indeed,  we have
\begin{gather*}
    \sum_{b\in\ZZ_n^{d}}|f(b)|C^{n}(a,b)
    \leq
    \rVert f\lVert_\infty\sup_{a\in\ZZ_n^{d}}\sum_{b\in\ZZ_n^{d}}C^{n}(a,b),\\
    \sum_{b\in\ZZ_n^{d}}|f(b)|C^{n}(a,b)
    \leq
    \sup_{a\in\ZZ_n^{d}}\sum_{b\in\ZZ_n^{d}}C^{n}(a,b)\sum_{b\in\ZZ_n^{d}}|f(b)|
    =
    \rVert f\lVert_{L^{1}_n}\sup_{a\in\ZZ_n^{d}}\sum_{b\in\ZZ_n^{d}}C^{n}(a,b)
\intertext{and, using the Cauchy--Schwarz inequality,}
\begin{aligned}
    \sum_{b\in\ZZ_n^{d}}|f(b)|C^{n}(a,b)
    &=\sum_{b\in\ZZ_n^{d}}|f(b)|\sqrt{C^{n}(a,b)}\sqrt{C^{n}(a,b)}\\
    &\leq\bigg(\sum_{b\in\ZZ_n^{d}}|f(b)|^{2}C^{n}(a,b)\bigg)^{1/2} \bigg(\sum_{b\in\ZZ_n^{d}}C^{n}(a,b)\bigg)^{1/2}\\
    &\leq\rVert f\rVert_{L_n^2}\sup_{a\in\ZZ_n^{d}}\sum_{b\in\ZZ_n^{d}}C^{n}(a,b).
\end{aligned}
\end{gather*}

We are interested in conditions which ensure the convergence of the family $\{\mathbf{X}^n\}_{n\in\nat}$ as $n\to\infty$.

\subsection{Tightness}\label{sec21}

The proof of convergence relies on convergence criteria for semimartingales; our standard reference will be the monograph \cite{Jacod-Shiryaev-2003}. Let $\process{S}$ be a $d$-dimensional semimartingale on the stochastic basis $(\Omega,\mathcal{F},\process{\mathcal{F}},\mathbb{P})$, and denote by $h:\real^{d}\to\real^{d}$ a truncation function, i.e.\ a bounded and continuous function which such that $h(x)=x$ in a neighbourhood of the origin. Since a semimartingale has c\`adl\`ag (right-continuous, finite left limits) paths, we can write $\Delta S_t := S_t - S_{t-}$, $t>0$, and $\Delta S_0 := S_0$, for the jumps of $S$ and set
$$
    \bar{S}(h)_t
    :=\sum_{s\leq t}(\Delta S_s-h(\Delta S_s))
    \et
    S(h)_t:=S_t-\bar{S}(h)_t.
$$
The process $\process{S(h)}$ is a special semimartingale, i.e. it admits a unique decomposition
\begin{equation}\label{eq1.1}
    S(h)_t=S_0+M(h)_t+B(h)_t,
\end{equation}
where $\process{M(h)}$ is a local martingale and $\process{B(h)}$ is a predictable process of bounded variation on compact time-intervals.

\begin{definition}
    Let $\process{S}$  be a semimartingale and $h:\real^{d}\to\real^{d}$ be a truncation function. The \emph{characteristics} of the semimartingale (relative to the truncation $h$) is a triplet $(B,A,N)$ consisting of the bounded variation process $B = \process{B(h)}$ appearing in \eqref{eq1.1}, the compensator $N=N(\omega,ds,dy)$ of the jump measure
    $$
        \mu(\omega,ds,dy):=\sum_{s:\Delta S_s(\omega)\neq 0}\delta_{(s,\Delta S_s(\omega))}(ds,dy)
    $$
    of the semimartingale $\process{S}$ and the quadratic co-variation process
    $$
        A_t^{ik}= \langle S^{i,c}_t,S^{k,c}_t\rangle,\quad i,k=1,\dots, d,\: t\geq 0,
    $$
    of the continuous part $\process{S^c}$ of the semimartingale.

    The \emph{modified characteristics} is the triplet $(B,\tilde{A},N)$ where $\tilde{A}(h)^{ik}_t:=\langle M(h)^{i}_t,M(h)^{k}_t\rangle_{L^2}$, $i,k=1,\dots,d$, with $\process{M(h)}$ being the local martingale appearing in \eqref{eq1.1}.
\end{definition}

Using \cite[Proposition II.2.17 and Theorem II.2.42]{Jacod-Shiryaev-2003} we can easily obtain the (modified) characteristics of $\process{X^{n}}$ from the infinitesimal generator; as before, we write $h:\real^{d}\to\real^{d}$ for the truncation function:
\begin{align*}
    B^{n}(h)_t
    &=\int_0^{t}\sum_{b\in\ZZ_n^{d}}h(b)C^{n}(X_s^{n},X_s^{n}+b)\,ds,\\
    \tilde{A}^{n}(h)_t^{ik}
    &=\int_0^{t}\sum_{b\in\ZZ_n^{d}}h_{i}(b)h_{k}(b)C^{n}(X_s^{n},X_s^{n}+b)\,ds,\\
    N^{n}(ds,b)
    &=C^{n}(X_s^{n},X_s^{n}+b)\,ds
\end{align*}
and, since $\process{X^n}$ is purely discontinuous, $A^n\equiv 0$.

In order to show the tightness of the family $\process{X^{n}}$, $n\in\nat$´, we need further conditions:
\begin{gather}
    \label{t3}\tag{\textbf{T3}} \forall \rho>0 \::\: \limsup_{n\to\infty}\sup_{a\in\ZZ_n^{d}}\sum_{|b|>\rho}C^{n}(a,a+b)<\infty;\\
    \label{t4}\tag{\textbf{T4}} \lim_{r\to\infty}\limsup_{n\to\infty}\sup_{a\in\ZZ_n^{d}}\sum_{|b|>r}C^{n}(a,a+b)=0;\\
    \label{t5}\tag{\textbf{T5}} \exists \rho>0\: \forall i=1,\dots,d\::\: \limsup_{n\to\infty}\sup_{a\in\ZZ_n^{d}}
                                    \bigg|\sum_{|b|<\rho}b_{i}C^{n}(a,a+b)\bigg|<\infty;\\
    \label{t6}\tag{\textbf{T6}} \exists \rho>0 \: \forall i,k=1,\dots,d\::\: \limsup_{n\to\infty}\sup_{a\in\ZZ_n^{d}}
                                    \bigg|\sum_{|b|<\rho}b_{i}b_{k}C^{n}(a,a+b)\bigg|<\infty.
\end{gather}

\begin{theorem}\label{tm1.2}
    Assume that \eqref{t1}--\eqref{t6} are satisfied. If the family of initial distributions $\mathbb{P}^n(X_0^n\in\bullet)$ is tight, then the family of Markov chains $\process{X^{n}}$, $n\in\nat$, is tight.
\end{theorem}
\begin{proof}
We denote by  $\mathbb{P}^{n}$ the law of $\process{X^{n}}$ such that $\{\mathbb{P}^{n}(X_0^{n}\in\bullet)\}_{n\geq1}$ is tight. According to \cite[Theorem VI.4.18]{Jacod-Shiryaev-2003}, the family  $\process{X^{n}}$, $n\in\nat$, will be tight, if for all $T>0$ and all $\varepsilon>0$,
\begin{align}\label{eq1.2}
    \lim_{r\to\infty}\limsup_{n\to\infty}
    \mathbb{P}^{n}\left(N^{n}([0,T],\{a\in\ZZ_n^{d}:\, |a|>r\})>\varepsilon\right) = 0,
\end{align}
and the families of processes $\process{B^{n}(h)}$, $\process{\tilde{A}^{n}(h)}$ and $\left\{\int_0^{t}\sum_{a\in\ZZ_n^{d}}g(a)N^{n}(ds,a)\right\}_{t\geq0}$, $n\in\nat$, are tight for every bounded function $g:\real^{d}\to\real$ which vanishes in a neighbourhood of the origin.

Clearly, \eqref{eq1.2} is a direct consequence of \eqref{t4}, and it is enough to show the tightness of the families $\process{B^{n}(h)}$, $\process{\tilde{A}^{n}(h)}$ and $\left\{\int_0^{t}\sum_{a\in\ZZ_n^{d}}g(a)N^{n}(ds,a)\right\}_{t\geq0}$, $n\in\nat$.

According to \cite[Theorem VI.3.21]{Jacod-Shiryaev-2003} tightness of $\process{B^{n}(h)}$, $n\in\nat$, follows if we can show that
\begin{itemize}
\item[(i)]
    for every $T>0$  there exists some $r>0$ such that
    $$
        \lim_{n\to\infty}\mathbb{P}^{n}\bigg(\sup_{t\in[0,T]}|B^{n}(h)_t|>r\bigg)=0;
    $$
\item[(ii)]
    for every $T>0$ and $r>0$  there exists some $\tau>0$ such that
    $$
        \lim_{n\to\infty}\mathbb{P}^{n}\bigg(\sup_{u,v\in [0,T], |u-v|\leq\tau}|B^{n}(h)_u-B^{n}(h)_v|>r\bigg)=0.
    $$
\end{itemize}
Fix $T>0$; without loss of generality we may assume that $h(x)=x$ for all $x\in B_\rho(0)\subset \real^{d}$ where $\rho>0$ is given in \eqref{t5}. For $i=1,\dots,d$, we have
\begin{align*}
    &\sup_{t\in[0,T]}|B^{n}(h)_t^{i}|
    =\sup_{t\in[0,T]}\bigg|\int_0^{t}\sum_{b\in\ZZ_n^{d}}h_{i}(b)C^{n}(X_s^{n},X_s^{n}+b)\bigg|\\
    &\qquad\leq\sup_{t\in[0,T]}\int_0^{t}\bigg|\sum_{|b|<\rho}b_{i}C^{n}(X_s^{n},X_s^{n}+b)\bigg| \,ds
      + \sup_{t\in[0,T]}\int_0^{t}\sum_{|b|\geq\rho}|h_{i}(b)|C^{n}(X_s^{n},X_s^{n}+b)\,ds\\
    &\qquad\leq T\sup_{a\in\ZZ_n^{d}}\bigg|\sum_{|b|<\rho}b_{i}C^{n}(a,a+b)\bigg|
      + T\|h\|_{\infty}\sup_{a\in\ZZ_n^{d}}\sum_{|b|\geq\rho}C^{n}(a,a+b)
\end{align*}
and  for all $0\leq u\leq v\leq T$ such that $|u-v|\leq\tau$ we get
\begin{align*}
    \big|B^{n}(h)_u^{i}-B^{n}(h)_v^{i}\big|
    &= \bigg|\int_{u}^{v}\sum_{b\in\ZZ_n^{d}}h_{i}(b)C^{n}(X_s^{n},X_s^{n}+b)\,ds\bigg|\\
    &\leq \tau\sup_{a\in\ZZ_n^{d}}\bigg|\sum_{|b|<\rho}b_{i}C^{n}(a,a+b)\bigg|
      +\tau\|h\|_{\infty}\sup_{a\in\ZZ_n^{d}}\sum_{|b|\geq\rho}C^{n}(a,a+b).
\end{align*}
The assertion now follows from \eqref{t3} and \eqref{t5}.

Since the proof of tightness of the other two families is very similar, we omit the details; note that these proofs require the (not yet used) conditions \eqref{t4} and \eqref{t6}.
\end{proof}

\subsection{On a Class of Jump Processes and their Dirichlet Forms}\label{sec22}
In order to identify the (weak) limit of the family $\process{X^{n}}$, $n\in\nat$, we will use Dirichlet forms. We restrict  ourselves to a class of pure jump processes whose infinitesimal generators have the following form
$$
    \mathcal{A}f(x):=\lim_{\varepsilon\downto 0}\int_{B^{c}_\varepsilon(x)}(f(y)-f(x))k(x,y)\,dy,\quad f\in\mathcal{D}_{\mathcal{A}}\subseteq L^{2}(\real^{d},dx),
$$
where $k:\real^{d}\times\real^{d}\setminus\diag\to [0,\infty)$ is a Borel measurable function defined off the diagonal $\diag:=\{(x,x):x\in\real^{d}\}$. Observe that many interesting  processes fall into this class. For instance, (non-)symmetric L\'evy processes generated by a L\'evy measure of the form $\nu(dy)=\nu(y)dy$ (here, $k(x,y)=\nu(y-x)$). But this class goes beyond L\'evy processes; for example it  contains a process generated by $k(x,y)=|x-y|^{-\alpha(x)-d}$, where $\alpha:\rd\to(0,2)$,  the so-called stable-like processes (in the sense of R.F.\ Bass \cite{Bass-1988}), cf.\ Example \ref{e1.18}.

Denote by $k_s(x,y):=\frac{1}{2}(k(x,y)+k(y,x))$ and $k_a(x,y):=\frac{1}{2}(k(x,y)-k(y,x))$ the symmetric and antisymmetric parts of $k(x,y)$, respectively. It is well known, cf.\ \cite[Example 1.2.4]{Fukushima-Oshima-Takeda-Book-2011}, that the assumption
\begin{gather*}
    x\mapsto\int_{\real^{d}}(1\wedge|y-x|^{2})k_s(x,y)\,dy\in L_{\mathrm{loc}}^{1}(\real^{d},dx)
\end{gather*}
ensures that $k(x,y)$ defines a regular symmetric Dirichlet form $(\mathcal{E},\mathcal{F})$ on $L^{2}(\real^{d},dx)$, where
\begin{align*}
    \mathcal{E}(f,g)  &:=\int_{\real^{d}\times\real^{d}\setminus \diag}(f(y)-f(x))(g(y)-g(x))k(x,y)\,dx\,dy,\quad f,g\in\mathcal{\bar{F}},\\
    \mathcal{\bar{F}} &:=\{f\in L^{2}(\real^{d},dx):\mathcal{E}(f,f)<\infty\}.
\end{align*}
The form domain $\mathcal{F}$ is the $\mathcal{E}^{1/2}_1$-closure of the Lipschitz continuous functions with compact support $C_c^{\mathrm{Lip}}(\real^{d})$ 
and $\mathcal{E}_\alpha(f,f):=\mathcal{E}(f,f)+\alpha\|f\|^{2}_{L^{2}}$, $\alpha>0$,  induces  a norm on $\mathcal{\bar{F}}$. The fact that $\mathcal E(f,g)=\mathcal E(g,f)$ holds, allows us to replace $k(x,y)$ in the definition of the form with its symmetric part $k_s(x,y)$. This explains why there is no condition on $k_a(x,y)$.

In order to deal with the non-symmetric setting we need a slightly stronger assumption
\begin{equation}\label{c1.h}\tag{\textbf{C1}}\begin{gathered}
    x\mapsto\int_{\real^{d}}(1\wedge|y-x|^{2})k_s(x,y)\,dy\in L_{\mathrm{loc}}^{1}(\real^{d},dx)\\
    \alpha_0
    := \sup_{x\in\real^{d}}\int_{\{y\in\real^{d}:\,k_s(x,y)\neq0\}}\frac{k_a(x,y)^{2}}{k_s(x,y)}\,dy < \infty.
\end{gathered}\end{equation}
Under this assumption, see \cite{Fukushima-Uemura-2012} and \cite{Schilling-Wang-2015}, the non-symmetric bilinear form
$$
    H(f,g)
    := -\lim_{\varepsilon\downto 0}\int_{\real^{d}}\int_{B^{c}_\varepsilon(x)}(f(y)-f(x))k(x,y)\,dy\,g(x)\,dx,
    \quad f,g\in C_c^{\mathrm{Lip}}(\real^{d}),
$$
is well defined and it has the representation
$$
    H(f,g)
    = \frac{1}{2}\mathcal{E}(f,g)
      - \int_{\real^{d}\times\real^{d}\setminus \diag} (f(y)-f(x))g(y) k_a(x,y)\,dx\,dy,
      \quad f,g\in C_c^{\mathrm{Lip}}(\real^{d}).
$$
Moreover, $H$ has an extension onto $\mathcal{F}\times\mathcal{F}$ such that $(H,\mathcal{F})$ defines a regular lower bounded  semi-Dirichlet form on $L^{2}(\real^{d},dx)$ in the sense of \cite{Ma-Rockner-Book-1992}; in particular, there is a properly associated Hunt process $(\process{X}, \{\mathbb{P}^{x}\}_{x})$ which is defined up to an exceptional 
set. The proofs of \cite[Theorem 2.1]{Fukushima-Uemura-2012} and \cite[Proposition 2.1]{Schilling-Wang-2015} reveal that
\begin{gather}
\label{eq-comp1}
    \frac{1}{4}(1\wedge\alpha_0) \mathcal{E}_1(f,f)
    \leq H_{\alpha_0}(f,f)
    \leq \frac{2+\sqrt{2}}{2}(1\vee\alpha_0)\,\mathcal{E}_1(f,f),
    \quad f\in \mathcal{F},
\intertext{and}
\label{eq-comp2}
    (1\wedge\alpha_0)H_1(f,f)
    \leq H_{\alpha_0}(f)
    \leq (1\vee\alpha_0)\,H_1(f),
    \quad f\in \mathcal{F},
\end{gather}
where $H_\alpha(f):=H(f,f)+\alpha\|f\|^{2}_{L^{2}}$, $\alpha>0$. In particular, $\mathcal{F}$ is also the $H^{1/2}_1$-closure of $C_c^{\mathrm{Lip}}(\real^{d})$ in $\mathcal{\bar{F}}$. Note that $H=\mathcal{E}$ if $\alpha_0=0$.

A weaker version of the following result already appears in \cite[Theorem 2.4]{Schilling-Uemura-2012}, the main difference is that in \cite{Schilling-Uemura-2012} the shift-boundedness of $k_s(x,y)$ was needed.


\begin{proposition}\label{p1.3}
    Under \eqref{c1.h} the family $C_c^{\infty}(\real^{d})$ is dense in $\mathcal{F}$ with respect to $H^{1/2}_1$.
\end{proposition}
\begin{proof}
    Clearly, it suffices to prove that $C_c^{\infty}(\real^{d})$ is $\mathcal{E}^{1/2}_1$-dense in $C_c^{\mathrm{Lip}}(\real^{d})$.

    Let $\chi\in C_c^{\infty}(\real^{d})$ satisfying $0\leq\chi(x)\leq1$ for all $x\in \real^{d}$, $\supp\chi\subseteq \bar{B}_1(0)$ and $\int_{\real^{d}}\chi(x)\,dx=1$. For $\varepsilon>0$ set
    $$
        \chi_\varepsilon(x):=\varepsilon^{-d}\chi(x/\varepsilon),\quad x\in\real^{d}.
    $$
    By definition, $\chi_\varepsilon\in C_c^{\infty}(\real^{d})$, $\supp \chi_\varepsilon\subseteq \bar{B}_\varepsilon(0)$ and $\int_{\real^{d}}\chi_\varepsilon(x)\,dx=1$. The Friedrichs mollifier $J_{1/n}$ of $g\in C_c^{\mathrm{Lip}}(\real^{d})$ is defined as
    $$
        J_{1/n}g(x)
        =\chi_{1/n}\ast g(x)
        =n^{d}\int_{\real^{d}}\chi_{1/n}(x-y)g(y)\,dy
        =\int_{\bar{B}_1(0)}\chi(y)g(x-y/n)\,dy,
    $$
    $x\in\real^{d}$; for brevity, we use $g_n := J_{1/n}g$. It is easy to see that $\{g_n\}_{n\geq1}$ converges uniformly to $g(x)$, $\|g_n\|_\infty\leq\|g\|_\infty$ and $\supp g_n\subseteq \supp g + \bar{B}_1(0)$ for all $n\in\nat$; in particular, $\{g_n\}_{n\geq1}$ converges to $g(x)$ in $L^{2}(\real^{d},dx)$. Hence, it remains to prove that
    $$
        \lim_{n\to\infty}\mathcal{E}(g_n-g,g_n-g)=0.
    $$
    First, observe that for all $n\in\nat$ and $x,y\in\real^{d}$
    $$
        |g_n(x)-g_n(y)|
        \leq\int_{\bar{B}_1(0)}\chi(z)|g(x-z/n)-g(y-z/n)|\,dz
        \leq L_g|x-y|,
    $$
    where $L_g>0$ is the Lipschitz constant of $g(x)$. Pick $R>0$ such that $\supp g_n\subseteq \bar{B}_R(0)$ for all $n\in\nat$. Then, we have that
    \begin{align*}
        &(g_n(y)-g_n(x))^{2}\I_{\real^{d}\times\real^{d}\setminus\diag}(x,y)\\
        &=(g_n(y)-g_n(x))^{2}\left(\I_{B_{2R}(0)\times B^{c}_{2R}(0)} + \I_{B^{c}_{2R}(0)\times B_{2R}(0)} + \I_{B_{2R}(0)\times B_{2R}(0)\setminus\diag}\right)(x,y)\\
        &\leq g_n^2(x)\I_{B_{2R}(0)\times B_{2R}^c(0)}(x,y)
        + g_n^2(y)\I_{B_{2R}^c(0)\times B_{2R}(0)}(x,y)
        + L_g^2 |x-y|^2 \I_{B_{2R}(0)\times B_{2R}(0)}(x,y)\\
        &\leq \|g\|_\infty^2\I_{B_{2R}(0)\times B_{2R}^c(0)}(x,y)
        + \|g\|_\infty^2\I_{B_{2R}^c(0)\times B_{2R}(0)}(x,y)
        + L_g^2 |x-y|^2 \I_{B_{2R}(0)\times B_{2R}(0)}(x,y).
    \end{align*}
    Since
    \begin{align*}
        \mathcal{E}(g_n-g,g_n-g)
        &=\int_{\real^{d}\times\real^{d}\setminus\diag}(g_n(y)-g(y)-g_n(x)+g(x))^{2}k_s(x,y)\,dx\,dy\\
        &\leq\int_{\real^{d}\times\real^{d}\setminus\diag}\left(2(g_n(y)-g_n(x))^{2}+2(g(y)-g(x))^{2}\right)k_s(x,y)\,dx\,dy,
    \end{align*}
    the assertion follows directly from \eqref{c1.h} and the dominated convergence theorem.
\end{proof}

Let us now describe the Dirichlet form related to the Markov chains $\process{X^{n}}$, $n\in\nat$, introduced in the previous Section~\ref{sec21}. first, recall that for $n\in\nat$, we denote by $L^{2}_n$ the standard Hilbert space on $\ZZ_n^{d}$ with scalar product
$$
    \langle f,g\rangle_{L^2_n}
    :=n^{-d}\sum_{a\in\ZZ_n^{d}}f(a)g(a),
    \quad f,g\in L^{2}_n.
$$
The following result is a direct consequence of \cite{Fukushima-Uemura-2012} and \cite{Schilling-Wang-2015}.
\begin{proposition}\label{p1.4}
    Assume that $C^{n}$, $n\in\nat$, satisfy \eqref{t1}, \eqref{t2} and \eqref{c1.h}\footnote{\eqref{c1.h} is assumed to hold for each chain $\process{X^{n}}$ by replacing the kernel $k(x,y)$ by $C^{n}(a,b)$}. For every $n\in\nat$ we define the bilinear forms
	\begin{align*}
        \mathcal{E}^{n}(f,g)
        &:=n^{-d}\sum_{a\in\ZZ_n^{d}}\sum_{b\in\ZZ_n^{d}}(f(b)-f(a))(g(b)-g(a))C_s^{n}(a,b)\\
        H^{n}(f,g)
        &:=\frac{1}{2}\mathcal{E}^{n}(f,g)-n^{-d}\sum_{a\in\ZZ_n^{d}}\sum_{b\in\ZZ_n^{d}}(f(b)-f(a))g(b)C^{n}_a(a,b).
    \end{align*}
    where $C^{n}_s(a,b):=\frac{1}{2}(C^{n}(a,b)+C^{n}(b,a))$, resp., $C^{n}_a(a,b):=\frac{1}{2}(C^{n}(a,b)-C^{n}(b,a))$ are the symmetric and antisymmetric parts of $C^{n}(a,b)$.
	\begin{itemize}
	\item[\textup{(i)}]
        $H^n(f,g)$ is a well defined non-symmetric bilinear form on $\mathcal{F}^{n}:=\{f\in L^{2}_n: \mathcal{E}^{n}(f,f)<\infty\}$.
	\item[\textup{(ii)}] $(H^{n},\mathcal{F}^{n})$ is  a regular lower bounded  semi-Dirichlet form \textup{(}in the sense of \cite{Ma-Rockner-Book-1992}\textup{)};
	\item[\textup{(iii)}]
    $H^{n}(f,g)=\langle-\mathcal{A}^{n}f,g\rangle_{L^2_n}$
    for all $f\in L_n^2$ and $g\in\mathcal F^n$.
    In particular, the associated Hunt process is $\process{X^{n}}$. Recall that
    $$
    \mathcal{A}^{n}f(a)
    =\sum_{b\in\ZZ_n^{d}}(f(b)-f(a))C^{n}(a,b)\quad \text{and}\quad L_n^2\subseteq\mathcal{D}_{\mathcal{A}^{n}}.
    $$
    \item[\textup{(iv)}]
        The estimates \eqref{eq-comp1} and \eqref{eq-comp2} hold for $\mathcal E=\mathcal E^n$, $H=H^n$ and
		$$
            \alpha_0 = \alpha^{n}_0:=\sup_{a\in\ZZ^{d}_n}\sum_{\substack{b\in\ZZ_n^{d}\\C^{n}_s(a,b)\neq0}}\frac{C^{n}_a(a,b)^{2}}{C^{n}_s(a,b)}.
        $$

	\end{itemize}
\end{proposition}

\begin{corollary}\label{c1.5}
    Assume that $C^{n}$, $n\in\nat$, satisfy \eqref{t1}, \eqref{t2} as well as
    \begin{gather}
    \label{t2*}\tag{\textbf{T$\mathbf{2^{*}}$}}
    \sup_{b\in\ZZ_n^{d}}\sum_{a\in\ZZ_n^{d}}C^{n}(a,b)<\infty.
    \end{gather}
    Then $\mathcal{F}^{n}=L^{2}_n$ for every $n\in\nat$.
\end{corollary}
\begin{proof}
	Clearly, \eqref{t2} and \eqref{t2*} imply \eqref{c1.h}. The claim follows from Jensen's and H\"older's inequalities.
\end{proof}

In order to study the convergence of the forms $H^n$ as $n\to\infty$ we need a few further notions. Denote by $\bar{a}:=[a_{1}-1/2n,a_{1}+1/2n)\times\cdots\times[a_{d}-1/2n,a_{d}+1/2n)$ the half-open cube with centre $a=(a_{1},\dots,a_{d})\in\mathbb{Z}_n^{d}$ and side-length $n^{-1}$, and for $x=(x_{1},\dots,x_{d})\in\real^{d}$ we set
$$
    [x]_n:=\left(\left[nx_1+1/2\right]/n,\dots,\left[nx_d+1/2\right]/n\right),
$$
where $[u]$ is the integer part of $u\in\real$. Note that for $a\in\ZZ_n^{d}$ and $x\in\bar{a}$ we have $[x]_n=a$.

\begin{figure}[H]
	\begin{center}
		\begin{tikzpicture}
		
		\draw[step=2cm,gray, thin] (-1,-1) grid (3,3);
		\filldraw[fill=black!15, draw=gray] (0,0) rectangle (2,2);
		\foreach \Point/\PointLabel in {(1,1)/}
		\draw[fill=black] \Point circle (0.03) node[above right] {$\PointLabel$};
		\draw[black,very thin] (1,1) circle (1.41cm);
		\draw[black,thin] (0,0) -- (1,1) node [midway, above, sloped] (TextNode) {\tiny$\sqrt{d}/2n$};
		\node at (0.5,1.6) (nodeA) {\Large $\bar{a}$};
		\node[] at (1,1.2) {\tiny$a\in\ZZ_n^{d}$};
		\draw[black,thin] (1,1) -- (1,0);
		\node[] at (1.35,0.5) {\tiny$1/2n$};
		\end{tikzpicture}
		\caption{Definition of the discretization}
	\end{center}
\end{figure}
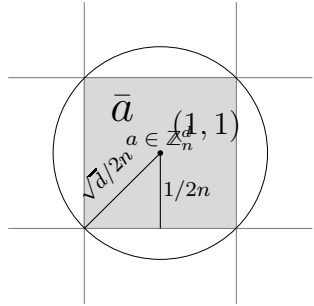

By $r_n:L^{2}(\real^{d},dx)\to L^{2}_n$ and $e_n:L^{2}_n\to L^{2}(\real^{d},dx)$ we denote the restriction and extension operators which are defined by
\begin{align*}
    r_nf(a) &= n^{d}\int_{\bar{a}}f(x)dx,\quad a\in\ZZ_n^{d}\\
    e_nf(x) &= f(a),\quad x\in\bar{a}.
\end{align*}
These operators have the following properties: for $f\in L^{2}(\real^{d},dx)$ and $f_n\in L^{2}_n$, $n\in\nat$:
\begin{itemize}
\item [(i)]

    $\displaystyle\sup_{n\in\nat}\|r_n\|_{L^2_n}\leq \|f\|_{L^2}$,
    $\displaystyle\lim_{n\to\infty}\|r_nf\|_{L^{2}_n}=\|f\|_{L^{2}}$ and
    $\|e_n f_n\|^{2}_{L^{2}}=\|e_n f_n^{2}\|_{L^{1}}=\|f_n\|_{L_n^{2}}^{2}$;
\item[(ii)]
    $r_n e_n f_n=f_n$ and $\langle r_n f,f_n\rangle_{L^2_n}=\langle f,e_nf_n\rangle_{L^2}$;
\item[(iii)]
    $\displaystyle\lim_{n\to\infty}\|e_nr_nf-f\|_{L^{2}}=0$;
\end{itemize}
see \cite[Lemma 4.1]{Chen-Kim-Kumagai-2013}. Let us recall  from \cite{Kuwae-Shioya-2003}  the notions of strong and weak convergence. Let $\mathcal{C}\subseteq L^{2}(\real^{d},dx)$ be dense in $(L^{2}(\real^{d},dx),\|\cdot\|_{L^{2}})$. A sequence $f_n\in L^{2}_n$, $n\in\nat$, \emph{converges strongly} to $f\in L^{2}(\real^{d},dx)$ if for every $\{g_m\}_{m\geq1}\subseteq \mathcal{C}$ satisfying
$$
    \lim_{m\to\infty}\|g_m-f\|_{L^{2}}=0,
$$
we have that
$$
    \lim_{m\to\infty}\limsup_{n\to\infty}\|r_ng_m-f_n\|_{L_n^{2}}=0.
$$
The sequence $f_n\in L^{2}_n$, $n\in\nat$, \emph{converges weakly} to $f\in L^{2}(\real^{d},dx)$, if
$$
    \lim_{n\to\infty}\langle f_n,g_n\rangle_{ L^2_n}=\langle f,g\rangle_{ L^2}
$$
for every sequence $\{g_n\}_{n\geq1}$, $g_n\in L^{2}_n$, converging strongly to $g\in L^{2}(\real^{d},dx)$.

In the following lemma we give an equivalent characterization of the strong and weak convergence, which simplifies the use of these types of convergence.
\begin{lemma}\label{l1.6}
\begin{itemize}
\item[\textup{(i)}]
    A  sequence $\{f_n\}_{n\geq1}$, $f_n\in L^{2}_n$, converges strongly to $f\in L^{2}(\real^{d},dx)$ if, and only if,
    $$
        \lim_{n\to\infty}\|e_nf_n-f\|_{L^{2}}=0.
    $$
\item [\textup{(ii)}]
    A  sequence $\{f_n\}_{n\geq1}$, $f_n\in L^{2}_n$, converges weakly to $f\in L^{2}(\real^{d},dx)$ if, and only if,
    $$
        \lim_{n\to\infty}\langle e_nf_n,g\rangle=\langle f,g\rangle,\quad g\in L^{2}(\real^{d},dx).
    $$
\end{itemize}
\end{lemma}
\begin{proof}
(i)
Let $f_n\in L^{2}_n$, $n\in\nat$, and $f\in L^{2}(\real^{d},dx)$. Assume that $\{f_n\}_{n\geq1}$ converges strongly  to $f$, and let $\{g_m\}_{m\geq1}\subseteq \mathcal{C}$ be an approximating sequence of $f$ satisfying the conditions from the definition of strong convergence. Then, by the properties of the operators $r_n$ and $e_n$, $n\in\nat$, we find
\begin{align*}
    \|e_nf_n-f\|_{L^{2}}
    &\leq\|e_nf_n-e_nr_nf\|_{L^{2}}+\|e_nr_nf-f\|_{L^{2}}\\
    &= \|f_n-r_nf\|_{L^{2}_n}+\|e_nr_nf-f\|_{L^{2}}\\
    &\leq \|f_n-r_ng_m\|_{L^{2}_n}+\|r_ng_m-r_nf\|_{L^{2}_n}+\|e_nr_nf-f\|_{L^{2}}\\
    &\leq \|f_n-r_ng_m\|_{L^{2}_n}+\|g_m-f\|_{L^{2}}+\|e_nr_nf-f\|_{L^{2}}.
\end{align*}
Letting first $n\to\infty$ and then $m\to\infty$, the necessity of the claim follows. For the `if' part, we proceed as follows.  Let $\{g_m\}_{m\geq1}\subseteq \mathcal{C}$ be any approximating sequence of $f$. Using the strong convergence $e_nf_n \to f$ and the fact that $r_ne_nf_n = f_n$, we see
\begin{align*}
    \|r_ng_m - f_n\|_{L^2_n}
    = \|r_ng_m - r_ne_n f_n\|_{L^2_n}
    &\leq \|g_m - e_nf_n\|_{L^2}\\
    &\leq \|g_m - f\|_{L^2} + \|f - e_nf_n\|_{L^2}.
\end{align*}
Letting first $n\to\infty$ and then $m\to\infty$ proves the assertion.

\medskip

(ii)
Let $f_n\in L^{2}_n$, $n\in\nat$, and $f\in L^{2}(\real^{d},dx)$. Assume that $\{f_n\}_{n\geq1}$ converges weakly to $f$. Since for every $g\in L^2(\real^{d})$ the sequence $\{r_ng\}_{n\geq1}$ converges strongly to $g$, it follows  immediately that
$$
    \lim_{n\to\infty}\langle e_nf_n,g\rangle_{L^2}
    = \langle f,g\rangle_{L^2},
    \quad g\in L^{2}(\real^{d},dx).
$$

For the sufficiency part we pick $g\in L^{2}(\real^{d},dx)$ and any sequence $\{g_n\}_{n\geq1}$, $g_n\in L^{2}_n$ such that $\{g_n\}_{n\geq1}$ converges strongly to $g$. By our assumption,
$$
    \lim_{n\to\infty}\langle f_n,r_ng\rangle_{L^2_n}
    = \lim_{n\to\infty}\langle e_nf_n,g\rangle_{L^2_n}
    = \langle f,g\rangle_{L^2};
$$
this shows, in particular, that $\langle e_nf_n, g\rangle_{L^2}\leq \|f\|_{L^2}$ for all $g$ with $\|g\|_{L^2}=1$, and so $\sup_{n\in\nat}\|e_nf_n\|_{L^2}<\infty$.
In order to prove the claim we have to show that
$$
    \lim_{n\to\infty}\langle f_n,g_n-r_ng\rangle_{L^2_n}=0.
$$
Using the properties of the operators $r_n$ and $e_n$ along with the Cauchy-Schwarz inequality yields
\begin{align*}|
    \langle f_n,g_n-r_ng\rangle_{L^2_n}|
    =|\langle f_n,r_ne_ng_n-r_ng\rangle_{L^2_n}|
    &=|\langle e_nf_n,e_ng_n-g\rangle_{L^2}|\\
    &\leq\|e_nf_n\|_{L^{2}}\|e_ng_n-g\|_{L^{2}},
\end{align*}
proving the assertion.
\end{proof}
For further details on strong and weak convergence we refer to \cite{Kuwae-Shioya-2003} and \cite{Tolle-Thesis-2006}.

\subsection{Convergence of the Finite-Dimensional Distributions}\label{sec23}

We can now combine the relative compactness from Section~\ref{sec21} and the convergence results from Section~\ref{sec22} to show the convergence of the  finite-dimensional distributions of the chains $\process{X^{n}}$, $n\in\nat$, to those of a non-symmetric pure jump process  $\process{X}$. The latter will be determined by a kernel $k:\real^{d}\times\real^{d}\setminus \diag\to\real$ satisfying \eqref{c1.h}.

\begin{theorem}\label{tm1.7}
    Assume that the chains $\process{X^{n}}$, $n\in\nat$, satisfy \eqref{t1}, \eqref{t2} and  \eqref{c1.h}. Let $\process{X}$ be a non-symmetric process determined by a kernel $k:\real^{d}\times\real^{d}\setminus \diag\to\real$ satisfying \eqref{c1.h}. Denote by $\process{P^{n}}$, $n\in\nat$, and $\process{P}$ the transition semigroups of $\process{X^{n}}$, $n\in\nat$, and $\process{X}$, respectively,

    If $\{P^{n}_tr_nf\}_{n\geq1}$ converges strongly to $P_tf$ for all $t\geq0$ and $f\in L^{2}(\real^{d},dx)$,  then there exists a Lebesgue null set $B$ such that the finite-dimensional distributions of $\process{X^{n}}$, $n\in\nat$, converge along $\mathbb{Q}$ on $B^{c}$ to those of $\process{X}$.
\end{theorem}
\begin{proof}
    Choose an arbitrary countable family $\mathcal{C}\subseteq C_c^{\mathrm{Lip}}(\real^{d})$ which is dense in $C^{\mathrm{Lip}}_c(\real^{d})$ with respect to $\|\cdot\|_\infty$. By the Markov property, the properties of the operators $r_n$ and $e_n$, and a standard diagonal argument, we can extract a subsequence $\{n_i\}_{i\geq 1}\subseteq\nat$ such that for all $m\geq 1$, all $t_1,\dots,t_m\in\mathbb{Q}$, $0\leq t_1\leq\dots\leq t_m<\infty$, and all $f_1,\dots f_m\in \mathcal{C}$, $\{\mathbb{E}^{\cdot}_n[r_nf_1(X^{n}_{t_1})\cdots r_nf_m(X^{n}_{t_m})]\}_{n\geq1}$ converges strongly to $\mathbb{E}^{\cdot}[f_1(X_{t_1})\cdots f_m(X_{t_m})]$.

    Again by a diagonal argument, we conclude that there is a further subsequence $\{n_{i}'\}_{i\geq1}\subseteq\{n_i\}_{i\geq1}$ and a Lebesgue null set $B\subseteq\real^{d}$, such that the above convergence holds pointwise on $B^{c}$. The assertion now follows from
    \cite[Proposition 3.4.4]{Ethier-Kurtz-Book-1986}.
\end{proof}

Denote by $\mathbb{D}(\real^{d})$ the  space of all c\`adl\`ag functions $f:[0,\infty)\to\real^{d}$. The space $\mathbb{D}(\real^{d})$  equipped  with Skorokhod's $J_1$ topology becomes a Polish space, the so-called Skorokhod space,  cf.\ \cite{Ethier-Kurtz-Book-1986}.
\begin{corollary}\label{p1.8}
    Assume that the conditions of Theorems \ref{tm1.2} and \ref{tm1.7} hold, and let $B$ be the Lebesgue null set from Theorem \ref{tm1.7}. Denote by $\mu_n$ and $\mu$ the initial distributions of $\process{X}$ and $\process{X^{n}}$, $n\in\nat$, respectively. If $\mu(B)=0$ and if $\mu_n\to\mu$ weakly,  then the following convergence holds in Skorokhod space:
	\begin{equation}\label{eq2.3}
        \left\{X^{n}_t\right\}_{t\geq0}\xrightarrow[n\to\infty]{d}\process{X}.
    \end{equation}
\end{corollary}
\begin{proof}
    The assertion follows by combining Theorems \ref{tm1.2}, \ref{tm1.7}, \cite[Lemma VI.3.19]{Jacod-Shiryaev-2003} and the remark following that lemma.
\end{proof}

Theorem \ref{tm1.7} states that \eqref{eq2.3} follows if we can prove  ``strong convergence" of  $P^n_t\to P_t$, $t>0$.
A sufficient condition for this convergence is given in the following theorem.
\enlargethispage{\baselineskip}
\begin{theorem}\label{tm1.9}
    Assume that \eqref{c1.h} holds for both $\process{X^{n}}$, $n\in\nat$, and $\process{X}$. The semigroups $\{P^{n}_tr_nf\}_{n\geq1}$ converge strongly to $P_tf$ for all $t\geq0$ and $f\in L^{2}(\real^{d},dx)$ if the following conditions are satisfied:
    \begin{gather}
    \label{c2.h}\tag{\textbf{C2}}
        0<\liminf_{n\to\infty}\alpha_0^{n}\leq\limsup_{n\to\infty}\alpha_0^{n}<\infty;\\
    \label{c3.h}\tag{\textbf{C3}}
        \forall\rho>0\::\:
        \sup_{x\in B_\rho(0)}\int_{\real^{d}}(1\wedge|y|^{2})k_s(x,x+y)\,dy<\infty;\\
    \label{c4.h}\tag{\textbf{C4}}
        \forall\rho>0\::\:
        \limsup_{n\to\infty}\sup_{a\in B_\rho(0)}\sum_{b\in\ZZ_n^{d}}(1\wedge|b|^{2})C_s^{n}(a,a+b)<\infty;\\
    \label{c5.h}\tag{\textbf{C5}}
        \forall\varepsilon>0\:\:\exists n_0\in\nat\:\:\forall n_0\leq m\leq n,\; f\in L^{2}_m\::\: \\
    \notag
        \mathcal{E}^{n}(r_ne_mf,r_ne_mf)^{^{1/2}}\leq \mathcal{E}^{m}(f,f)^{1/2}+\varepsilon;\\
    \label{c6.h}\tag{\textbf{C6}}
        \parbox{.85\textwidth}{for all sufficiently small $\varepsilon > 0$ and large $m\in\nat$}\\
        \lim_{n\to\infty} \bar{\mathcal{E}}^{n}_{m,\varepsilon}(f,f)
        =\mathcal{E}_{m,\varepsilon}(f,f),\quad f\in C^{\mathrm{Lip}}_c(\real^{d})\notag
    \intertext{where for all $f\in\Cclip(\real^d)$
    $$\begin{aligned}
    \mathcal{E}_{m,\varepsilon}(f,f)
        &:=\frac{1}{2}\iint_{\{(x,y)\in B_m(0)\times B_m(0):\, |x-y|>\varepsilon\}}(f(y)-f(x))^{2}k_s(x,y)\,dx\,dy,\\
    \bar{\mathcal{E}}^{n}_{m,\varepsilon}(f,f)
        &:=\frac{n^{d}}{2}\iint_{\{(x,y)\in B_m(0)\times B_m(0):\, |x-y|>\varepsilon\}}(f(y)-f(x))^{2} C^{n}_s(a,b)\I_{\bar a\times\bar b}(x,y)\,dx\,dy,
    \end{aligned}$$
    }
    \label{c7.h}\tag{\textbf{C7}}
    \begin{aligned}[t]
    \textup{(i)}&\qquad x\mapsto\int_{B_1(0)}|y|^{2}k_s(x,x+y)\,dy\in L_{\mathrm{loc}}^{2}(\real^{d},dx),\\
        \textup{(ii)}&\qquad x\mapsto\int_{B^{c}_1(0)}k_s(x,x+y)\,dy\in L^{2}(\real^{d},dx)\cup L^{\infty}(\real^{d},dx),\\
        \textup{(iii)}&\qquad x\mapsto\int_{B_1(0)}|y||k_s(x,x+y)-k_s(x,x-y)|\,dy\in L_{\mathrm{loc}}^{2}(\real^{d},dx),\\
        \textup{(iv)}&\qquad x\mapsto\int_{\rd}(1\wedge|y|) |k_a(x,x+y)|\,dy\in L_{\mathrm{loc}}^{2}(\real^{d},dx);
    \end{aligned}\\
    \label{c8.h}\tag{\textbf{C8}}
        \int_{\rd} (1\wedge|y|^2)|k_s(x,x+y)-n^{d}C^{n}_s([x]_n,[x]_n+[y]_n)|\,dy
        \xrightarrow[n\to\infty]{L_{\mathrm{loc}}^{2}(\real^{d},dx)}0;\\
    \label{c9.h}\tag{\textbf{C9}}
        \parbox{.85\textwidth}{for all sufficiently large $R>1$}\\
        \int_{B^{c}_{2R}(0)}\bigg(\int_{B_R(-x)}k_s(x,x+y)\,dy\bigg)^{2}dx<\infty;\notag\\
    \label{c10.h}\tag{\textbf{C10}}
        \parbox{.85\textwidth}{for all sufficiently large $R>1$}\\
        \int_{B^{c}_{2R}(0)}\bigg(\int_{B_R(-x)}|k_s(x,x+y)-n^{d}C^{n}_s([x]_n,[x]_n+[y]_n)|\,dy\bigg)^{2}dx\xrightarrow{n\to\infty}0;\notag\\
    \label{c11.h}\tag{\textbf{C11}}
        \int_{B_1(0)}|y||k_s(x,x+y)-k_s(x,x-y)\phantom{-n^{d}C^{n}_s([x]_n,[x]_n+[y]_n)+n^{d}C^{n}_s([x]_n,[x]_n-[y]_n)|}\\
        \qquad\mbox{}-n^{d}C^{n}_s([x]_n,[x]_n+[y]_n)+n^{d}C^{n}_s([x]_n,[x]_n-[y]_n)|\,dy
        \xrightarrow[n\to\infty]{L_{\mathrm{loc}}^{2}(\real^{d},dx)}0;\notag\\
    \label{c12.h}\tag{\textbf{C12}}
        \int_{\rd}(1\wedge|y|)|k_a(x,x+y)-n^{d}C^{n}_a([x]_n,[x]_n+[y]_n)|\,dy
        \xrightarrow[n\to\infty]{L_{\mathrm{loc}}^{2}(\real^{d},dx)}0;\\
    \label{c13.h}\tag{\textbf{C13}}
                \parbox{.85\textwidth}{for all sufficiently large $R>1$}\\
        \int_{B^{c}_{2R}(0)}\bigg(\int_{B_R(-x)}|k_a(x,x+y)-n^{d}C^{n}_a([x]_n,[x]_n+[y]_n)|\,dy\bigg)^{2}dx
        \xrightarrow{n\to\infty}0.\notag
    \end{gather}
\end{theorem}
\begin{proof}
According to the properties of the operators $r_n$ and $e_n$, $n\in\nat$, Propositions \ref{p1.3} and \ref{p1.4}, and \cite[Theorem 2.41 and Remark 2.44]{Tolle-Thesis-2006} the assertion follows if
\begin{itemize}
  \item[(i)]
    for every sequence $\{f_n\}_{n\geq1}$, $f_n\in \mathcal{F}^{n}$, converging weakly to some $f\in L^{2}(\real^{d},dx)$ and satisfying $\displaystyle\liminf_{n\to\infty}H^{n}_1(f_n,f_n)<\infty$, we have that $f\in\mathcal{F}$;

  \item[(ii)]
    for every $g\in C_c^{2}(\real^{d})$ and every sequence $\{f_n\}_{n\geq1}$, $f_n\in \mathcal{F}^{n}$, converging weakly to $f\in \mathcal{F}$,
    $$
        \lim_{n\to\infty}H^{n}(r_ng,f_n)=H(g,f).
    $$
\end{itemize}
Indeed, (i) and (ii) imply that for any sequence $\{f_n\}_{n\geq1}$, $f_n\in L^{2}_n$, converging strongly to $f\in L^{2}(\real^{d},dx)$, the sequence $\{P^{n}_tf_n\}_{n\geq1}$ converges strongly to $P_tf$ for every $t\geq0$, cf.\ \cite[Theorem 2.41 and Remark 2.44]{Tolle-Thesis-2006}.
 For any fixed $f\in L^2(\rd,dx)$ we set $f_n = r_n f$. Since $r_n f \to f$ strongly we conclude that $P_t^n r_nf \to P_tf$ strongly as claimed.

Let us now prove that \eqref{c1.h}--\eqref{c13.h} imply (i) and (ii). We begin with (i). According to Proposition \ref{p1.4}, we have
$$
    \mathcal{E}^{n}_1(f)\leq \frac{4(1\vee\alpha^{n}_0)}{1\wedge\alpha^{n}_0}H^{n}_{1}(f),
    \quad f\in \mathcal{F}^{n}.
$$
Let  $\{f_n\}_{n\geq1}$, $f_n\in \mathcal{F}^{n}$,  be an arbitrary sequence converging weakly to some $f\in L^{2}(\real^{d},dx)$ such that $\liminf_{n\to\infty}H^{n}_1(f_n)<\infty$; the condition \eqref{c2.h} ensures that $\liminf_{n\to\infty}\mathcal{E}^{n}_1(f_n)<\infty$. Finally, \cite[Theorem 4.6]{Chen-Kim-Kumagai-2013} states that under \eqref{c3.h}--\eqref{c6.h}
$$
    \mathcal{E}(f,f)\leq\liminf_{n\to\infty}\mathcal{E}^{n}(f_n,f_n),
$$
which proves (i).

In order to prove (ii) we proceed as follows. According to Proposition \ref{p1.4} we have for  $g\in C_c^{2}(\real^{d})$ and  $f_n\in \mathcal{F}^{n}$
$$
    H^{n}(r_ng,f_n)
    =\langle-\mathcal{A}^{n}r_ng,f_n\rangle_{L^2_n},
    \quad n\in\nat.
$$
Using \eqref{c7.h} it is shown in \cite[Theorem 2.2]{Schilling-Wang-2015} that the generator $(\mathcal{A},\mathcal{D}_{\mathcal{A}})$
of $\process{P}$ (or, equivalently, of $(H,\mathcal{F})$) has the following properties:
\begin{itemize}
\item[(i)]
    $C_c^{2}(\real^{d})\subseteq\mathcal{D}_\mathcal{A}$;
\item[(ii)]
    for every $g\in C_c^{2}(\real^{d})$,
    \begin{align*}
        \mathcal{A}g(x) &= \int_{\real^{d}}(g(x+y)-g(x)-\langle\nabla g(x),y\rangle\I_{B_1(0)}(y))k_s(x,x+y)\,dy\\
        &\quad+ \frac{1}{2}\int_{B_1(0)}\langle\nabla g(x),y\rangle(k_s(x,x+y)-k_s(x,x-y))\,dy\\
        &\quad+\int_{\real^{d}}(g(x+y)-g(x))k_a(x,x+y)\,dy;
    \end{align*}
\item[(iii)]
    for all $g\in C_c^{2}(\real^{d})$ and all $f\in\mathcal{F}$,
    $$
        H(g,f)=\langle-\mathcal{A}g,f\rangle_{L^2}.
    $$
\end{itemize}
Therefore, it suffices to prove that $\{\mathcal{A}^{n}r_ng\}_{n\geq1}$ converges strongly to $\mathcal{A}g$ for every $g\in C^{2}_c(\real^{d})$. Observe that for any $g\in C_c^{2}(\real^{d})$ and $n\in\nat$,
\begin{align*}
    \mathcal{A}^{n}r_ng(a)
    &=\sum_{b\in\ZZ_n^{d}}(r_ng(a+b)-r_ng(a))C^{n}(a,a+b)\\
    &=\sum_{b\in\ZZ_n^{d}}(r_ng(a+b)-r_ng(a)-\langle r_n\nabla g(a),b\rangle\I_{\{|b|\leq1\}}(b))C^{n}_s(a,a+b)\\
    &\quad+ \frac{1}{2}\sum_{|b|\leq1}\langle r_n\nabla g(a),b\rangle(C^{n}_s(a,a+b)-C^{n}_s(a,a-b))\\
    &\quad+\sum_{b\in\ZZ_n^{d}}(r_ng(a+b)-r_ng(a))C^{n}_a(a,a+b).
\end{align*}
Using the triangle inequality we get
$$
    \|e_n\mathcal{A}^{n}r_ng-\mathcal{A}g\|_{L^{2}}
    \leq \sum_{i=1}^5\|A^{n}_i-A_i\|_{L^{2}},
$$
where for $\rho := 1 + (2n)^{-1}\sqrt{d}$ the $A_i$ and $A_i^n$ are given by
\begin{align*}
    A_1
    &:=\int_{B_{\rho}(0)} \big(g(x+y)-g(x)-\langle\nabla g(x),y\rangle \I_{B_1(0)}(y)\big) \, k_s(x,x+y)\,dy\\
    A_1^{n}
    &:=\int_{B_{\rho}(0)}\big(r_ng([x]_n+[y]_n)-r_ng([x]_n)-\langle r_n\nabla g([x]_n),[y]_n\rangle \I_{B_1(0)}([y]_n)\big)\times\\
    &\hphantom{\int_{B_{\rho}(0)}\big(r_ng([x]_n+[y]_n)-r_ng([x]_n)-\langle r_n\nabla g([x]_n),[y]_n\rangle}\mbox{}\times n^{d}\,C^{n}_s([x]_n,[x]_n+[y]_n)\,dy\\
    A_2
    &:=\int_{B^{c}_{\rho}(0)}\big(g(x+y)-g(x)\big) \, k_s(x,x+y)\,dy\\
    A_2^{n}
    &:=\int_{B^{c}_{\rho}(0)}\big(r_ng([x]_n+[y]_n)-r_ng([x]_n)\big) \,n^{d}\, C^{n}_s([x]_n,[x]_n+[y]_n)\,dy\\
    A_3
    &:=\frac{1}{2}\int_{B_{\rho}(0)}\langle\nabla g(x),y\rangle \I_{B_1(0)}(y) (k_s(x,x+y)-k_s(x,x-y))\,dy\\
    A^{n}_3
    &:=\frac{1}{2}\int_{B_{\rho}(0)}\langle r_n\nabla g([x]_n),[y]_n\rangle\I_{B_1(0)}([y]_n)\times\\
    &\hphantom{\frac{1}{2}\int_{B_{\rho}(0)}\langle r_n\nabla g([x]_n),[y]_n\rangle}\times n^{d}\, (C^{n}_s([x]_n,[x]_n+[y]_n)-C^{n}_s([x]_n,[x]_n-[y]_n))\,dy\\
    A_4
    &:=\int_{B_{\rho}(0)}\big(g(x+y)-g(x)\big) \, k_a(x,x+y)\,dy\\
    A_4^{n}
    &:=\int_{B_{\rho}(0)}\big(r_ng([x]_n+[y]_n)-r_ng([x]_n)\big) \,n^{d} \,C^{n}_a([x]_n,[x]_n+[y]_n)\,dy\\
    A_5
    &:=\int_{B^{c}_{\rho}(0)}\big(g(x+y)-g(x)\big)\,k_a(x,x+y)\,dy\\
    A_5^{n}
    &:=\int_{B^{c}_{\rho}(0)}\big(r_ng([x]_n+[y]_n)-r_ng([x]_n)\big) \,n^{d}\, C^{n}_a([x]_n,[x]_n+[y]_n)\,dy.
\end{align*}
  In the remaining part of the proof we assume $R>1+\sqrt{d}/2$ such that $\supp g\subseteq B_R(0)$ and we write $\rho := 1 + (2n)^{-1}\sqrt{d}$ and $\sigma := 1 - (2n)^{-1}\sqrt{d}$.
\begin{align*}
    &\|A^{n}_1-A_1\|_{L^{2}}\\
    &\leq \bigg(\int_{B_{2R}(0)}\bigg|\int_{B_{\rho}(0)} \big[ g(x+y)-g(x)-\langle\nabla g(x),y\rangle\I_{B_1(0)}(y)\\
    &\qquad\mbox{} -r_ng([x]_n+[y]_n) + r_ng([x]_n) + \langle r_n\nabla g([x]_n),[y]_n\rangle\I_{B_1(0)}([y]_n)\big] \, k_s(x,x+y)\, dy\bigg|^{2}dx\bigg)^{\frac{1}{2}}\\
    &\qquad\mbox{} + \bigg(\int_{B_{2R}(0)}\bigg|\int_{B_{\rho}(0)} \big[r_ng([x]_n+[y]_n)-r_ng([x]_n)-\langle r_n\nabla g([x]_n),[y]_n\rangle\I_{B_1(0)}([y]_n)\big]\times\\
    &\qquad\qquad\qquad\qquad\qquad\qquad\qquad\qquad \times\big(k_s(x,x+y)-n^{d}C^{n}_s([x]_n,[x]_n+[y]_n)\big)\, dy\bigg|^{2} dx\bigg)^{\frac{1}{2}}\\
    &\leq
    \bigg(\int_{B_{2R}(0)}\bigg|\int_{B_{\sigma}(0)} \big[ g(x+y)-g(x)-\langle\nabla g(x),y\rangle\\
    &\qquad\qquad\mbox{}-r_ng([x]_n+[y]_n) + r_ng([x]_n) + \langle r_n\nabla g([x]_n),[y]_n\rangle\big] \, k_s(x,x+y)\,dy\bigg|^{2} dx\bigg)^{\frac{1}{2}}\\
    &\qquad\mbox{} + \big(4\|g\|_\infty+2\rho\|\nabla g\|_\infty\big) \bigg(\int_{B_{2R}(0)}\bigg(\int_{B_{\rho}(0)\setminus B_{\sigma}(0)} k_s(x,x+y)\,dy\bigg)^{2} dx\bigg)^{\frac{1}{2}}\\
    &\qquad\mbox{}+ \|\nabla^{2}g\|_\infty \bigg(\int_{B_{2R}(0)}\bigg(\int_{B_{\sigma}(0)}|y|^{2}\, \big| k_s(x,x+y)-n^{d}\,C^{n}_s([x]_n,[x]_n+[y]_n)\big|\,dy \bigg)^{2} dx\bigg)^{\frac{1}{2}}\\
    &\qquad\mbox{}+ \big(2\|g\|_\infty+\rho\|\nabla g\|_\infty\big) \times\\
    &\qquad\qquad\mbox{}\times\bigg(\int_{B_{2R}(0)}\bigg(\int_{B_{\rho}(0)\setminus B_{\sigma}(0)}\big|k_s(x,x+y)-n^{d}\,C^{n}_s([x]_n,[x]_n+[y]_n)\big|\,dy\bigg)^{2} dx\bigg)^{\frac{1}{2}}.
\end{align*}
By     monotone and dominated convergence theorem, Taylor's theorem, \eqref{c7.h} (i) and (ii), and \eqref{c8.h}, we conclude  that $\|A^{n}_1-A_1\|_{L^{2}}\to 0$.
Next,
\begin{align*}
    &\|A^{n}_2-A_2\|_{L^{2}}\\
    &\leq \bigg(\int_{B_{2R}(0)}\bigg|\int_{B^{c}_{\rho}(0)} \big[g(x+y)-g(x)-r_ng([x]_n+[y]_n)+r_ng([x]_n)\big] \, k_s(x,x+y)\, dy\bigg|^{2} dx\bigg)^{\frac{1}{2}}\\
    &\qquad\mbox{} + \bigg(\int_{B^{c}_{2R}(0)}\bigg|\int_{B^{c}_{\rho}(0)} \big[ g(x+y)-r_ng([x]_n+[y]_n)\big] \, k_s(x,x+y)\, dy\bigg|^{2} dx\bigg)^{\frac{1}{2}}\\
    &\qquad\mbox{} + \bigg(\int_{B_{2R}(0)}\bigg|\int_{B^{c}_{\rho}(0)} \big[ r_ng([x]_n+[y]_n)-r_ng([x]_n)\big] \times\\
    &\qquad\qquad\qquad \times\big(k_s(x,x+y)-n^{d}C^{n}_s([x]_n,[x]_n+[y]_n)\big) \,dy\bigg|^{2} dx\bigg)^{\frac{1}{2}}\\
    &\qquad\mbox{}+\bigg(\int_{B^{c}_{2R}(0)} \bigg|\int_{B^{c}_{\rho}(0)} r_ng([x]_n+[y]_n) \big(k_s(x,x+y)-n^{d}C^{n}_s([x]_n,[x]_n+[y]_n)\big) \,dy\bigg|^{2} dx\bigg)^{\frac{1}{2}}\\
    &\leq \bigg(\int_{B_{2R}(0)} \bigg|\int_{B^{c}_{\rho}(0)} \big[g(x+y)-g(x)-r_ng([x]_n+[y]_n) + r_ng([x]_n)\big]\, k_s(x,x+y)\, dy\bigg|^{2} dx\bigg)^{\frac{1}{2}}\\
    &\qquad\mbox{} +\bigg(\int_{B^{c}_{2R}(0)} \bigg(\int_{B_{R}(-x)} \big|g(x+y)-r_ng([x]_n+[y]_n)\big| \,k_s(x,x+y) \,dy\bigg)^{2} dx\bigg)^{\frac{1}{2}}\\
    &\qquad\mbox{} +2\|g\|_\infty \bigg(\int_{B_{2R}(0)}\bigg(\int_{B^{c}_{\rho}(0)} \big|k_s(x,x+y)-n^{d}C^{n}_s([x]_n,[x]_n+[y]_n)\big| \,dy\bigg)^{2} dx\bigg)^{\frac{1}{2}}\\
    &\qquad\mbox{} +\|g\|_\infty\bigg(\int_{B^{c}_{2R}(0)} \bigg(\int_{B_{R}(-x)} \big|k_s(x,x+y)-n^{d}C^{n}_s([x]_n,[x]_n+[y]_n)\big| \,dy\bigg)^{2} dx\bigg)^{\frac{1}{2}}.
\end{align*}

Again, by monotone and dominated convergence theorem, Taylor's theorem, \eqref{c7.h} (ii), \eqref{c8.h}, \eqref{c9.h} and \eqref{c10.h}, we have  that $\|A^{n}_2-A_2\|_{L^{2}}\to 0$. Further,
\begin{small}
\begin{align*}
&\|A^{n}_3-A_3\|_{L^{2}}\\
    &\leq\frac{1}{2}\bigg(\int_{B_{2R}(0)}\bigg|\int_{B_{\rho}(0)}\big[\langle\nabla g(x),y\rangle\I_{B_1(0)}(y)-\langle r_n\nabla g([x]_n),[y]_n\rangle\I_{B_1(0)}([y]_n)\big]\times\\
    &\qquad\qquad\qquad\qquad\qquad\qquad\qquad\qquad \mbox{}\times\big(k_s(x,x+y)-k_s(x,x-y)\big)\,dy\bigg|^{2}dx\bigg)^{\frac{1}{2}}\\
    &\quad\mbox{}+\frac{1}{2}\bigg(\int_{B_{2R}(0)}\bigg|\int_{B_{\rho}(0)}\langle r_n\nabla g([x]_n),[y]_n\rangle\I_{B_1(0)}([y]_n)\times\\
    &\quad\mbox{}\times\big(k_s(x,x+y)-k_s(x,x-y)-n^{d}C^{n}_s([x]_n,[x]_n+[y]_n)+n^{d}C^{n}_s([x]_n,[x]_n-[y]_n)\big) \,dy\bigg|^{2}dx\bigg)^{\frac{1}{2}}\\
    &\leq\frac{1}{2}\bigg(\int_{B_{2R}(0)}\bigg|\int_{B_{\sigma}(0)}\big[\langle\nabla g(x),y\rangle-\langle r_n\nabla g([x]_n),[y]_n\rangle\big]\big(k_s(x,x+y)-k_s(x,x-y)\big)\,dy\bigg|^{2}dx\bigg)^{\frac{1}{2}}\\
    &\quad\mbox{} + \rho \|\nabla g\|_\infty \bigg(\int_{B_{2R}(0)}\bigg(\int_{B_{\rho}(0)\setminus B_{\sigma}(0)} \big|k_s(x,x+y)-k_s(x,x-y)\big| \,dy\bigg)^{2}dx\bigg)^{\frac{1}{2}}\\
    &\quad\mbox{} + \|\nabla g\|_\infty\bigg(\int_{B_{2R}(0)}\bigg(\int_{B_{\sigma}(0)}|y|\times\\
    &\quad\mbox{}\times \big|k_s(x,x+y)-k_s(x,x-y)-n^{d}C^{n}_s([x]_n,[x]_n+[y]_n)+n^{d}C^{n}_s([x]_n,[x]_n-[y]_n)\big| \,dy\bigg)^{2} dx\bigg)^{\frac{1}{2}}\\
    &\quad\mbox{}+\frac{1}{2} \rho \|\nabla g\|_\infty\bigg(\int_{B_{2R}(0)}\bigg(\int_{B_{\rho}(0)\setminus B_{\sigma}(0)}\\
    &\quad\mbox{}\times \big|k_s(x,x+y)-k_s(x,x-y)-n^{d}C^{n}_s([x]_n,[x]_n+[y]_n)+n^{d}C^{n}_s([x]_n,[x]_n-[y]_n)\big| \,dy\bigg)^{2} dx\bigg)^{\frac{1}{2}},
\end{align*}%
\end{small}%
which, by  monotone and dominated convergence,  Taylor's theorem, \eqref{c7.h} (iii), \eqref{c8.h} and \eqref{c11.h}, implies  that $\|A^{n}_3-A_3\|_{L^{2}}\to 0$. Next,
\begin{align*}
&\|A^{n}_4-A_4\|_{L^{2}}\\
    &\leq\bigg(\int_{B_{2R}(0)}\bigg|\int_{B_{\rho}(0)} \big[g(x+y)-g(x)-r_ng([x]_n+[y]_n)+r_ng([x]_n)\big] k_a(x,x+y)\,dy\bigg|^{2} dx\bigg)^{\frac{1}{2}}\\
    &\qquad\mbox{}+\bigg(\int_{B_{2R}(0)}\bigg|\int_{B_{\rho}(0)} \big[r_ng([x]_n+[y]_n)-r_ng([x]_n)\big]\times\\
    &\qquad\qquad\qquad\qquad\qquad\qquad\mbox{}\times\big(k_a(x,x+y)-n^{d}C^{n}_a([x]_n,[x]_n+[y]_n)\big) \,dy\bigg|^{2} dx\bigg)^{\frac{1}{2}}\\
&\leq\bigg(\int_{B_{2R}(0)}\bigg|\int_{B_{\rho}(0)} \big[g(x+y)-g(x)-r_ng([x]_n+[y]_n)+r_ng([x]_n)\big] k_a(x,x+y)\,dy\bigg|^{2} dx\bigg)^{\frac{1}{2}}\\
    &\qquad\mbox{}+2\|\nabla g\|_\infty \bigg(\int_{B_{2R}(0)}\bigg(\int_{B_{\rho}(0)} |y| \big|k_a(x,x+y)-n^{d}C^{n}_a([x]_n,[x]_n+[y]_n)\big| \,dy\bigg)^{2} dx\bigg)^{\frac{1}{2}}.
\end{align*}
Now, by  monotone and dominated convergence,  Taylor's theorem, \eqref{c7.h} (iv) and \eqref{c12.h}, we see $\|A^{n}_4-A_4\|_{L^{2}}\to 0$. Finally,
\begin{align*}
&\|A^{n}_5-A_5\|_{L^{2}}\\
    &\leq\bigg(\int_{B_{2R}(0)}\bigg|\int_{B^{c}_{\rho})(0)} \big[g(x+y)-g(x)-r_ng([x]_n+[y]_n)+r_ng([x]_n)\big] k_a(x,x+y)\,dy\bigg|^{2} dx\bigg)^{\frac{1}{2}}\\
    &\qquad\mbox{}+ \bigg(\int_{B^{c}_{2R}(0)}\bigg|\int_{B^{c}_{\rho}(0)} \big[g(x+y)-r_ng([x]_n+[y]_n)\big] k_a(x,x+y) \,dy\bigg|^{2} dx\bigg)^{\frac{1}{2}}\\
    &\qquad\mbox{}+\bigg(\int_{B_{2R}(0)}\bigg|\int_{B^{c}_{\rho}(0)} \big[r_ng([x]_n+[y]_n)-r_ng([x]_n)\big]\times\\
    &\qquad\qquad\qquad\qquad\qquad\mbox{}\times\big(k_a(x,x+y)-n^{d}C^{n}_a([x]_n,[x]_n+[y]_n)\big)\, dy\bigg|^{2} dx\bigg)^{\frac{1}{2}}\\
    &\qquad\mbox{}+\bigg(\int_{B^{c}_{2R}(0)}\bigg|\int_{B^{c}_{\rho}(0)}r_ng([x]_n+[y]_n) \big(k_a(x,x+y)-n^{d}C^{n}_a([x]_n,[x]_n+[y]_n)\big) \,dy\bigg|^{2} dx\bigg)^{\frac{1}{2}}\\
&\leq\bigg(\int_{B_{2R}(0)}\bigg|\int_{B^{c}_{\rho}(0)} \big[g(x+y)-g(x)-r_ng([x]_n+[y]_n)+r_ng([x]_n)\big] k_a(x,x+y)\,dy\bigg|^{2} dx\bigg)^{\frac{1}{2}}\\
    &\qquad\mbox{} + \bigg(\int_{B^{c}_{2R}(0)}\bigg(\int_{B_R(-x)}|g(x+y)-r_ng([x]_n+[y]_n)||k_a(x,x+y)|\,dy\bigg)^{2}dx\bigg)^{\frac{1}{2}}\\
    &\qquad\mbox{} + 2\|g\|_\infty \bigg(\int_{B_{2R}(0)}\bigg(\int_{B^{c}_{\rho}(0)} \big|k_a(x,x+y)-n^{d}C^{n}_a([x]_n,[x]_n+[y]_n)\big| \,dy\bigg)^2 dx\bigg)^{\frac{1}{2}}\\
    &\qquad\mbox{} + \|g\|_\infty \bigg(\int_{B^{c}_{2R}(0)}\bigg(\int_{B_R(-x)} \big|k_a(x,x+y)-n^{d}C^{n}_a([x]_n,[x]_n+[y]_n)\big| \,dy\bigg)^{2} dx\bigg)^{\frac{1}{2}}.
\end{align*}
By  monotone and dominated convergence,  Taylor's theorem, \eqref{c7.h} (iv), \eqref{c9.h}, \eqref{c12.h} and \eqref{c13.h}, we get  that $\|A^{n}_5-A_5\|_{L^{2}}\to 0$, which concludes the proof.
\end{proof}

The conditions of Theorem~\ref{tm1.9} can be slightly changed to give a further set of sufficient conditions of the convergence of $\{P^{n}_tr_nf\}_{n\geq1}$; the advantage is that we can state these conditions only using $k_s$ and $k$ resp.\ $C_s^n$ and $C^n$, which makes them sometimes easier to check.

\begin{corollary}\label{c1.10}
    Assume that \eqref{c1.h}--\eqref{c6.h} and \eqref{c7.h}\textup{(i)}, \textup{(ii)} hold, that
    \begin{equation}\label{eq-tm1.10}
        x\mapsto\int_{B_1(0)}|y|\big|k(x,x+y)-k(x,x-y)\big|\,dy\in L_{\mathrm{loc}}^{2}(\real^{d},dx)
    \end{equation}
    and that \eqref{c8.h}--\eqref{c11.h} hold with $k_s$ and $C_s^n$ replaced by $k$ and $C^n$, respectively. Then $\{P^{n}_tr_nf\}_{n\geq1}$ converges strongly to $P_tf$ for all $t\geq0$ and $f\in L^{2}(\real^{d},dx)$.
\end{corollary}

\begin{proof}
    According to \cite[Theorem 3.1]{Schilling-Wang-2015}, the above assumptions imply that the generator $(\mathcal{A},\mathcal{D}_{\mathcal{A}})$ of $(H,\mathcal{F})$ satisfies
    \begin{itemize}
        \item[(i)] $C_c^{\infty}(\real^{d})\subseteq\mathcal{D}_\mathcal{A}$;
        \item[(ii)] for every $g\in C_c^{\infty}(\real^{d})$,
        \begin{equation}\label{eq1.3}\begin{aligned}
        \mathcal{A}g(x)
        &= \int_{\real^{d}} \big(g(x+y)-g(x)-\langle\nabla g(x),y\rangle\I_{B_1(0)}(y)\big) k(x,x+y)\,dy\nonumber\\
        &\quad\mbox{}+ \frac{1}{2}\int_{B_1(0)}\langle\nabla g(x),y\rangle \big(k(x,x+y)-k(x,x-y)\big)\,dy.
        \end{aligned}\end{equation}
    \end{itemize}
     Note that in \cite[Theorem 3.1]{Schilling-Wang-2015} slightly stronger conditions are assumed (namely (H3) which is a symmetrized version of \eqref{eq-tm1.10} and the tightness assumption (H5)), but they are exclusively used to deal with the formal adjoint $A^*$; this follows easily from an inspection of the proofs of \cite[Theorems 2.2 and 3.1]{Schilling-Wang-2015}.

    From this point onwards we can follow the proof of Theorem \ref{tm1.9}.
\end{proof}

    Recall that a set $\mathcal{C}\subseteq \mathcal{D}_{\mathcal{A}}$ is an operator core for $(\mathcal{A},\mathcal{D}_{\mathcal{A}})$ if $\overline{\mathcal{A}|_{\mathcal{C}}}=\mathcal{A}$. If we happen to know that $C_c^{2}(\real^{d})$ is an operator core for $(\mathcal{A},\mathcal{D}_{\mathcal{A}})$, then there is an alternative proof of Theorem~\ref{tm1.9} and its Corollary~\ref{c1.10} based on
    \cite[Theorem 1.6.1]{Ethier-Kurtz-Book-1986}: $\{P^{n}_tr_nf\}_{n\geq1}$ converges strongly to $P_tf$ for all $t\geq0$ and all $f\in L^{2}(\real^{d},dx)$ if (and only if) $\{\mathcal{A}^{n}r_ng\}_{n\geq1}$ converges strongly to $\mathcal{A}g$ for every $g\in C^{2}_c(\real^{d})$.

\subsection{Approximation  of a Given Process}\label{sec24}

We will now show how we can use the results of Sections~\ref{sec21}--\ref{sec23} to approximate a given non-symmetric pure-jump process by a sequence of Markov chains. We assume that $\process{X}$ is of the type described at the beginning of Section~\ref{sec22}; in particular the kernel $k:\real^{d}\times\real^{d}\setminus \diag\to\real$ satisfies \eqref{c1.h}. We are going to construct a sequence of approximating (in the weak sense) Markov chains. Let $0<p\leq1$ and define a family of  kernels  $C^{n,p}:\ZZ_n^{d}\times\ZZ_n^{d}\to[0,\infty)$, $n\in\nat$, by
\begin{equation}\label{eq1.5}
    C^{n,p}(a,b):=
    \begin{cases}
        \displaystyle n^{d}\int_{\bar{a}}\int_{\bar{b}} k(x,y)\,dx\,dy, & |a-b|>\frac{2\sqrt{d}}{n^{p}} \\
        \displaystyle 0, & |a-b|\leq\frac{2\sqrt{d}}{n^{p}}.
    \end{cases}
\end{equation}

\begin{remark}\label{r1.11}
Th family of kernels defined in \eqref{eq1.5} has the following properties:
\begin{itemize}
\item[(i)]  The kernels $C^{n,p}$, $n\in\nat$, automatically satisfy \eqref{t1}.
\item[(ii)] For any increasing sequence $\{n_i\}_{i\in\nat}\subset\nat$ such that the lattices are nested, i.e.\ $\ZZ^{d}_{n_i}\subseteq\ZZ^{d}_{n_{i+1}}$, the conditions \eqref{c5.h}  and \eqref{c6.h}  hold true, cf.~\cite[Theorem 5.4]{Chen-Kim-Kumagai-2013}. This is, in particular, the case for $n_i = 2^i$, $i\in\nat$.
\item[(iii)] Due to \eqref{c1.h} and Lebesgue's differentiation theorem (see \cite[Theorem 3.21]{Folland-Book-1984}),
    we have for (Lebesgue) almost all $(x,y)\in\real^{d}\times\real^{d}\setminus\diag$,
    \begin{gather*}
        \lim_{n\to\infty} n^{2d} \int_{\overline{[x]}_n}\int_{\overline{[y]}_n} k(u,v)\,dv\,du = k(x,y)
    \intertext{and}
        \lim_{n\to\infty}n^{2d} \int_{\overline{[x]}_n}\int_{\overline{[y]}_n} |k(u,v)-k(x,y)|\,dv\,du=0.
    \end{gather*}
\end{itemize}
\end{remark}

Let us check the conditions \eqref{t2}--\eqref{t6}.
\begin{proposition}\label{p1.11}
The conditions \eqref{t2} and \eqref{t3} hold true if
\begin{gather}
\label{t1.d}\tag{\textbf{T1.D}}
    \forall\rho>0\::\: \sup_{x\in\real^{d}}\int_{B^{c}_{\rho}(x)}k(x,y)dy<\infty.
\end{gather}
\end{proposition}
\begin{proof}
    We will only discuss \eqref{t2} since \eqref{t3} follows in a similar way. Observe that for every $d\in\nat$ and
    $0<p\leq 1$,
    $$
        \bigcup_{|a-b|>2\sqrt{d}/n^{p}}\bar{b}
        \subseteq B^{c}_{\sqrt{d}/n^{p}}(a)
        \subseteq B^{c}_{\sqrt{d}/2n^{p}}(x),
        \quad a\in\ZZ_n^{d},\ x\in\bar{a}.
    $$
    This shows that
    \begin{align*}
        \sup_{a\in\ZZ_n^{d}}\sum_{b\in\ZZ_n^{d}}C^{n,p}(a,b)
        &= n^{d}\sup_{a\in\ZZ_n^{d}}\sum_{|a-b|>2\sqrt{d}/n^{p}} \int_{\bar{a}}\int_{\bar{b}}k(x,y)\,dy\,dx\\
        &\leq n^{d}\sup_{a\in\ZZ_n^{d}}\int_{\bar{a}}\int_{B^{c}_{\sqrt{d}/2n^{p}}(x)}k(x,y)\,dy\,dx\\
        &\leq \sup_{x\in \real^{d}}\int_{B^{c}_{\sqrt{d}/2n^{p}}(x)}k(x,y)\,dy,
    \end{align*}
    which concludes the proof.
\end{proof}
This means that under \eqref{t1.d}, the kernels $C^{n,p}$, $n\in\nat$, define a family of regular Markov  chains  $\process{X^{n}}$, $n\in\nat$. Using the same arguments as above, it is easy to see that \eqref{t4} holds  if
\begin{gather}
    \label{t3.d}\tag{\textbf{T3.D}}
        \lim_{r\to\infty}\sup_{x\in\real^{d}}\int_{B^{c}_{r}(x)}k(x,y)dy=0.
\end{gather}

\begin{proposition}\label{p1.12}
    Assume that \eqref{t1.d} holds. Then the following statements are true.
    \begin{itemize}
    \item[\textup{(i)}] \eqref{t5} will be satisfied if
    \begin{align}\label{t4.d1}\tag{\textbf{T4.D.1}}
        &\text{there is some $\rho>0$  such that for $i=1,\dots, d$}\\
        &\begin{aligned}[t]
        &\limsup_{\varepsilon\downto 0}\sup_{x\in\real^{d}}\bigg|\int_{B_\rho(x)\setminus B_\varepsilon(x)}(y_{i}-x_{i})k(x,y)\,dy\bigg|<\infty,\\
        &\limsup_{\varepsilon\downto 0}\sup_{x\in\real^{d}}\int_{B_{\sqrt{d}\varepsilon^{p}}(x)\setminus B_{\sqrt{d}\varepsilon^{p}-(\sqrt{d}/2)\varepsilon}(x)}|y_i-x_i|k(x,y)\,dx<\infty,\\
        &\limsup_{\varepsilon\downto 0}\varepsilon\sup_{x\in\real^{d}}\int_{B_\rho(x)\setminus B_{\varepsilon^{p}}(x)}k(x,y)\,dy<\infty.
        \end{aligned}\notag
    \end{align}

    \item[\textup{(ii)}] \eqref{t6} will be satisfied if
    \begin{align}\label{t5.d1}\tag{\textbf{T5.D.1}}
        &\text{there is some $\rho>0$  such that for $i,k=1,\dots, d$}\\
        &\begin{aligned}[t]
            &\limsup_{\varepsilon\downto 0}\sup_{x\in\real^{d}}\bigg|\int_{B_\rho(x)\setminus B_\varepsilon(x)}(y_{i}-x_{i})(y_{k}-x_{k})k(x,y)\,dy\bigg|<\infty,\\
            &\limsup_{\varepsilon\downto 0}\sup_{x\in\real^{d}}\int_{B_{\sqrt{d}\varepsilon^{p}}(x)\setminus B_{\sqrt{d}\varepsilon^{p}-(\sqrt{d}/2)\varepsilon}(x)}|y_i-x_i||y_k-x_k|k(x,y)\,dx<\infty,\\
            &\limsup_{\varepsilon\downto 0}\varepsilon\sup_{x\in\real^{d}}\int_{B_\rho(x)\setminus B_{\varepsilon^{p}}(x)}|y_i-x_i|k(x,y)\,dy<\infty.
        \end{aligned}\notag
    \end{align}
    \end{itemize}
\end{proposition}
\begin{proof}
We will only discuss \eqref{t5}, since \eqref{t6} follows in an analogous way.  Assume \eqref{t4.d1}. We have
\begin{align*}
\sup_{a\in\ZZ_n^{d}}\bigg|\sum_{|b|<\rho}b_{i}\,C^{n,p}(a,a+b)\bigg|
    &=n^{d}\sup_{a\in\ZZ_n^{d}}\bigg|\sum_{2\sqrt{d}/n^{p}<|b|<\rho}b_{i}\int_{\bar{a}}\int_{\overline{a+b}}k(x,y)\,dy\,dx\bigg|\\
    &\leq n^{d}\sup_{a\in\ZZ_n^{d}}\bigg|\sum_{2\sqrt{d}/n^{p}<|b|<\rho}\int_{\bar{a}}\int_{\overline{a+b}}(y_i-x_i)k(x,y)\,dy\,dx\bigg|\\
    &\qquad\mbox{}+n^{d}\sup_{a\in\ZZ_n^{d}}\sum_{2\sqrt{d}/n^{p}<|b|<\rho}\int_{\bar{a}}\int_{\overline{a+b}}|b_i-y_i+x_i|k(x,y)\,dy\,dx\\
    &\leq n^{d}\sup_{a\in\ZZ_n^{d}}\bigg|\sum_{2\sqrt{d}/n^{p}<|b|<\rho}\int_{\bar{a}}\int_{\overline{a+b}}(y_i-x_i)k(x,y)\,dy\,dx\bigg|\\
    &\qquad\mbox{}+\sqrt{d}n^{d-1}\sup_{a\in\ZZ_n^{d}}\sum_{2\sqrt{d}/n^{p}<|b|<\rho}\int_{\bar{a}}\int_{\overline{a+b}}k(x,y)\,dy\,dx.
\end{align*}
Next,  for every $d\in\nat$, $0<p\leq1$ and all $a\in\ZZ_n^{d}$, $x\in\bar{a}$ we have
$$
    \bigcup_{2\sqrt{d}/n^{p}<|b|<\rho}\overline{a+b}
    \subseteq B_{\rho+\sqrt{d}/2n}(a)\setminus B_{\sqrt{d}/n^{p}}(a)
    \subseteq B_{\rho+\sqrt{d}/n}(x)\setminus B_{\sqrt{d}/n^{p}-\sqrt{d}/2n}(x).
$$
Thus,
\begin{align*}
&\sup_{a\in\ZZ_n^{d}}\bigg|\sum_{|b|<\rho}b_{i}\,C^{n,p}(a,a+b)\bigg|\\
    &\quad\leq n^{d}\sup_{a\in\ZZ_n^{d}}\bigg|\int_{\bar{a}}\int_{B_{\rho+\sqrt{d}/n}(x)\setminus B_{\sqrt{d}/n^{p}-\sqrt{d}/2n}(x)}(y_i-x_i)k(x,y)\,dy\,dx\bigg|\\
    &\quad\qquad\mbox{}+ n^{d}\sup_{a\in\ZZ_n^{d}}\bigg|\int_{\bar{a}}\bigg(\int_{B_{\rho+\sqrt{d}/n}(x)\setminus B_{\sqrt{d}/n^{p}-\sqrt{d}/2n}(x)}(y_i-x_i)k(x,y)\,dy\\
    &\quad\qquad\qquad\qquad\qquad\qquad\qquad\mbox{}-\sum_{2\sqrt{d}/n^{p}<|b|<\rho}\int_{\overline{a+b}}(y_i-x_i)k(x,y)\,dy\bigg)\,dx\bigg|\\
    &\quad\qquad\mbox{}+\sqrt{d}n^{d-1}\sup_{a\in\ZZ_n^{d}}\int_{\bar{a}}\int_{B_{\rho+\sqrt{d}/n}(x)\setminus B_{\sqrt{d}/n^{p}-\sqrt{d}/2n}(x)}k(x,y)\,dy\,dx
\\
    &\quad\leq \sup_{x\in\real^{d}}\bigg|\int_{B_{\rho+\sqrt{d}/n}(x)\setminus B_{\sqrt{d}/n^{p}-\sqrt{d}/2n}(x)}(y_i-x_i)k(x,y)\,dy\bigg|\\
    &\quad\qquad\mbox{}+ \sup_{x\in\real^{d}}\int_{B_{\rho+\sqrt{d}/n}(x)\setminus  B_{\rho-\sqrt{d}/n}(x)} |y_i-x_i|k(x,y)\,dy\\
    &\quad\qquad\mbox{}+ \sup_{x\in\real^{d}}\int_{ B_{\sqrt{d}/n^{p}}(x)\setminus B_{\sqrt{d}/n^{p}-\sqrt{d}/2n}(x)}|y_i-x_i|k(x,y)\,dy\\
    &\quad\qquad\mbox{}+\frac{\sqrt{d}}{n}\sup_{x\in\real^{d}}\int_{B_{\rho+\sqrt{d}/n}(x)\setminus B_{\sqrt{d}/n^{p}-\sqrt{d}/2n}(x)}k(x,y)\,dy
\\
    &\quad\leq \sup_{x\in\real^{d}}\bigg|\int_{B_{\rho+\sqrt{d}/n}(x)\setminus B_{\sqrt{d}/n^{p}-\sqrt{d}/2n}(x)}(y_i-x_i)k(x,y)\,dy\bigg|\\
    &\quad\qquad\mbox{}+ \bigg(\rho+\frac{\sqrt{d}}{n}\bigg)\sup_{x\in\real^{d}}\int_{B_{\rho+\sqrt{d}/n}(x)\setminus B_{\rho-\sqrt{d}/n}(x)}k(x,y)\,dy\\
    &\quad\qquad\mbox{}+\sup_{x\in\real^{d}}\int_{B_{\sqrt{d}/n^{p}}(x)\setminus B_{\sqrt{d}/n^{p}-\sqrt{d}/2n}(x)}|y_i-x_i|k(x,y)\,dx\\
    &\quad\qquad\mbox{}+\frac{\sqrt{d}}{n}\sup_{x\in\real^{d}}\int_{B_{\rho+\sqrt{d}/n}(x)\setminus B_{\sqrt{d}/n^{p}-\sqrt{d}/2n}(x)}k(x,y)\,dy.
\end{align*}
Together with \eqref{t1.d} and \eqref{t4.d1} this proves the claim.
\end{proof}

Let us now discuss the conditions \eqref{c2.h}--\eqref{c6.h}.

\begin{proposition}\label{p1.13}
\begin{itemize}
\item[\textup{(i)}]
    The condition \eqref{c1.h} implies
    $
        \limsup_{n\to\infty}\alpha_0^{n}\leq\alpha_0.
    $
\item[\textup{(ii)}]
    Assume that there exist open balls $B_1,B_2\subseteq\real^{d}$ such that $\sup_{(x,y)\in B_1\times B_2}k_s(x,y)<\infty$, $k_s(x,y)>0$ Lebesgue-a.e.\ on $B_1\times B_2$, and $\inf_{(x,y)\in B_1\times B_2}|k_a(x,y)|>0$. Then,
    $
        \liminf_{n\to\infty}\alpha_0^{n}>0.
    $
\end{itemize}
In particular, the assumptions in \textup{(i)} and \textup{(ii)} guarantee that \eqref{c2.h} holds true.
\end{proposition}
\begin{proof}
(i)
For $n\in\nat$, $0<p\leq1$ and $a,b\in\ZZ_n^{d}$ such that $C^{n,p}_s(a,b)\neq0$ the Cauchy--Schwarz inequality gives
\begin{align*}
    \bigg(\int_{\bar{a}}\int_{\bar{b}}|k_a(x,y)|dydx\bigg)^{2}
    &= \bigg(\int_{\bar{a}}\int_{\bar{b}}\I_{\{k_s\neq 0\}}(x,y) 
    \frac{|k_a(x,y)|}{\sqrt{k_s(x,y)}}\sqrt{k_s(x,y)}\,dy\,dx\bigg)^{2}\\
    &\leq \int_{\bar{a}}\int_{\bar{b}} \I_{\{k_s\neq 0\}}(x,y) 
    \frac{|k_a(x,y)|^{2}}{k_s(x,y)} \,dy\,dx\int_{\bar{a}}\int_{\bar{b}}k_s(x,y)\,dy\,dx,
\end{align*}
i.e.
$$
    \frac{C^{n,p}_a(a,b)^{2}}{C^{n,p}_s(a,b)}
    =\frac{\Big(n^{d}\int_{\bar{a}}\int_{\bar{b}}|k_a(x,y)|\,dy\,dx\Big)^{2}}{n^{d}\int_{\bar{a}}\int_{\bar{b}}k_s(x,y)\,dy\,dx}
    \leq n^{d}\int_{\bar{a}}\int_{\bar{b}} \I_{\{k_s\neq 0\}}(x,y) 
    \frac{|k_a(x,y)|^{2}}{k_s(x,y)}\,dy\,dx.
$$
Consequently,
\begin{align*}
    \alpha_0^{n}
    =\sup_{a\in\ZZ_n^{d}}\sum_{\substack{b\in\ZZ_n^{d}\\C^{n,p}_s(a,b)\neq0}}\frac{C^{n,p}_a(a,b)^{2}}{C^{n,p}_s(a,b)}
    &\leq\sup_{a\in\ZZ_n^{d}}\sup_{x\in\real^{d}}\sum_{\substack{b\in\ZZ_n^{d}\\C^{n,p}_s(a,b)\neq0}}
    \int_{\bar{b}} \I_{\{k_s\neq 0\}}(x,y) 
    \frac{|k_a(x,y)|^{2}}{k_s(x,y)}\,dy\\
    &\leq\sup_{x\in\real^{d}}\int_{\{y\in\real^{d}:\, k_s(x,y)\neq0\}}\frac{|k_a(x,y)|^{2}}{k_s(x,y)}\,dy
    =\alpha_0.
\end{align*}

\medskip
(ii) Let
$\displaystyle
    m := \inf_{(x,y)\in B_1\times B_2}|k_a(x,y)|
    \leq \sup_{(x,y)\in B_1\times B_2}k_s(x,y)
    =:M
$. For all $n\in\nat$ large enough, we have
\begin{gather*}
    \alpha_0^{n}
    =\sup_{a\in\ZZ_n^{d}}\sum_{\substack{b\in\ZZ_n^{d}\\C^{n,p}_s(a,b)\neq0}}\frac{C^{n,p}_a(a,b)^{2}}{C^{n,p}_s(a,b)}
    \geq\sup_{\substack{a\in\ZZ_n^{d}\\\bar{a}\subseteq B_1}}\sum_{\substack{b\in\ZZ_n^{d}\\\bar{b}\subseteq B_2}}\frac{C^{n,p}_a(a,b)^{2}}{C^{n,p}_s(a,b)}
    \geq \frac{m^{2}}{M} \sum_{\substack{b\in\ZZ_n^{d}\\\bar{b}\subseteq B_2}}n^{-d}.
\qedhere
\end{gather*}
\end{proof}

\begin{proposition}\label{p1.14}
    \eqref{c3.h} implies \eqref{c4.h}  for $C^{n,p}$ given by \eqref{eq1.5}.
\end{proposition}
\begin{proof}
For any $\rho>0$, $n\in\nat$ and $0<p\leq1$, we have
$$
    \sup_{a\in B_\rho(0)}\sum_{b\in\ZZ_n^{d}}(1\wedge|a-b|^{2})C^{n,p}_s(a,b)\leq\sup_{a\in B_\rho(0)}\sum_{|a-b|\geq1}C^{n,p}_s(a,b)+\sup_{a\in B_\rho(0)}\sum_{|a-b|<1}|a-b|^{2}C^{n,p}_s(a,b).
$$
Furthermore, since
$$
    \bigcup_{|a-b|\geq1}\bar{b}
    \subseteq B^{c}_{1-\sqrt{d}/2n}(a)
    \subseteq B^{c}_{1-\sqrt{d}/n}(x),
    \quad a\in\ZZ_n^{d},\ x\in\bar{a},
$$
for all $n\in\nat$ with $n>\sqrt{d}$, we have that
\begin{align*}
    \sup_{a\in B_\rho(0)}\sum_{|a-b|\geq1}C^{n,p}_s(a,b)
    &\leq \sup_{a\in B_\rho(0)}n^{d}\int_{\bar{a}}\int_{B^{c}_{1-\sqrt{d}/n}(x)}k_s(x,y)\,dy\,dx\\
    &\leq\sup_{a\in B_\rho(0)}\sup_{x\in\bar{a}}\int_{B^{c}_{1-\sqrt{d}/n}(x)}k_s(x,y)\,dy\\
    &\leq\sup_{x\in B_{\rho+\sqrt{d}/2n}(0)}\int_{B^{c}_{1-\sqrt{d}/n}(0)}k_s(x,x+y)\,dy.
\end{align*}
Next, it is easy to see that
$$
    \bigcup_{|a-b|<1}\bar{b}
    \subseteq B_{1+\sqrt{d}/2n}(a)
    \subseteq B_{1+\sqrt{d}/n}(x),
    \quad a\in\ZZ_n^{d},\ x\in\bar{a},
$$
and
$|a-b|\leq 2|x-y|$ 
for all $a,b\in\ZZ_n^{d}$, $|a-b|>\sqrt{d}/n$, and all $x,y\in\real^{d}$, $x\in\bar{a}$, $y\in\bar{b}$. Thus,
\begin{align*}
    \sup_{a\in B_\rho(0)}\sum_{|a-b|<1}|a-b|^{2}C^{n,p}_s(a,b)
    &=\sup_{a\in B_\rho(0)}\sum_{2\sqrt{d}/n^{p}<|a-b|<1}|a-b|^{2}C^{n,p}_s(a,b)\\
    &\leq 4\sup_{a\in B_\rho(0)}n^{d}\sum_{2\sqrt{d}/n^{p}<|a-b|<1}\int_{\bar{a}}\int_{\bar{b}}|x-y|^{2}k_s(x,y)\,dy\,dx\\
    &\leq4\sup_{a\in B_\rho(0)}n^{d}\int_{\bar{a}}\int_{B_{1+\sqrt{d}/n}(x)}|x-y|^{2}k_s(x,y)\,dy\,dx\\
    &\leq4\sup_{a\in B_\rho(0)}\sup_{x\in\bar{a}}\int_{B_{1+\sqrt{d}/n}(0)}|y|^{2}k_s(x,x+y)\,dy\\
    &\leq4\sup_{a\in B_{\rho+\sqrt{d}/2n}(0)}\int_{B_{1+\sqrt{d}/n}(0)}|y|^{2}k_s(x,x+y)\,dy,
\end{align*}
which proves the assertion.
\end{proof}

We will now discuss some examples where the conditions \eqref{c1.h}--\eqref{c13.h} are satisfied.
\begin{example}[Symmetric jump processes]\label{e1.16}
    Assume that $k(x,y)=k(y,x)$ Lebesgue a.e.\ on $\real^{d}\times\real^{d}$. For $0<p\leq1$ we define the corresponding family of conductances $C^{n,p}$, $n\in\nat$, by \eqref{eq1.5}. If \eqref{t1}--\eqref{t6} hold, then the family of underlying Markov  chains  $\{X^{n}_t\}_{t\geq0}$, $n\in\nat$, is tight. Due to symmetry,  the  second condition in \eqref{c1.h}, \eqref{c7.h} (iii), (iv), \eqref{c12.h} and \eqref{c13.h} are trivially satisfied, and \eqref{c2.h} is not needed. The condition \eqref{c4.h}  follows from \eqref{c3.h}, while \eqref{c5.h}  and \eqref{c6.h} automatically hold true (take, for example, a subsequence $n_i=2^{i}$, $i\in\nat$). For an alternative approach to the problem of discrete approximation of \emph{symmetric} jump processes we refer the readers to \cite{Chen-Kim-Kumagai-2013}.
\end{example}

\begin{example}[Non-symmetric L\'{e}vy processes]\label{e1.17}
    A class of non-symmetric L\'{e}vy processes which satisfy conditions  \eqref{c1.h}--\eqref{c13.h}  can be constructed in the following way. Let $\nu_1(dy)=n_1(y)\,dy$ and $\nu_2(dy)=n_2(y)\,dy$ be L\'{e}vy measures and let $B\subseteq\real^{d}$ be a Borel set. Define a new L\'{e}vy measure $\nu(dy)$ by
    $$
        \nu(dy) := \I_{B}(y)\,\nu_1(dy) + \I_{B^c}(y)\,\nu_d(dy).
    $$
    In general, $\nu(dy)$ is not symmetric and a suitable choice of the densities $n_1(y)$ and $n_2(y)$ ensures \eqref{c1.h}--\eqref{c13.h}. For example, take $n_1(y)=|y|^{-\alpha-d}\I_{B^c_1(0)}(y)$ and $n_2(y):=|y|^{-\beta-d}\I_{B^{c}_1(0)}(y)$, where $\alpha,\beta\in(0,2)$, and $B$ is any Borel set. Obviously, the kernel
    $$
        k(x,y) :=
		|y-x|^{-\alpha-d} \I_{B}(x-y) + |y-x|^{-\beta-d} \I_{B^c}(x-y),\quad x\neq y,
    $$
    satisfies  \eqref{c1.h} and the Dirichlet form $(H,\mathcal{F})$ corresponds to a pure jump  L\'{e}vy process with  L\'{e}vy measure $\nu(dy)$ defined as above. Finally, it is not hard to check that $k(x,y)$ satisfies the conditions in \eqref{c2.h}--\eqref{c13.h} (for a subsequence $n_i=2^{i}$, $i\in\nat$).
\end{example}

\begin{example}[Stable-like processes]\label{e1.18}
    Let $\alpha:\real^{d}\to(0,2)$ be a Borel measurable function. Consider the following integro-differential operator
    \begin{equation}\label{eq1.6}
        Lf(x)
        := \gamma(x)\int_{\real^{d}}\left( f(y+x)-f(x)-\langle\nabla f(x),y\rangle\I_{B_1(0)}(y)\right) \frac{dy}{|y|^{\alpha(x)+d}}
    \end{equation}
    where $f\in C^{\infty}_c(\real^{d})$ and
    $$
        \gamma(x)
        := \alpha(x)2^{\alpha(x)-1}\frac{\Gamma(\alpha(x)/2+d/2)}{\pi^{d/2}\Gamma(1-\alpha(x)/2)}
    $$
    ($\Gamma(x)$ is Euler's Gamma function). It is well known that
    $$
        Lf(x)
        = \int_{\real^{d}}e^{i\langle x,\xi\rangle}|\xi|^{\alpha(x)}\hat{f}(\xi)\,d\xi,\quad f\in C_c^{\infty}(\real^{d}),
    $$
    where $\hat{f}(\xi):= (2\pi)^{-d} \int_{\real^{d}} e^{-i\langle\xi,x\rangle} f(x)\, dx$ denotes the Fourier transform of the function $f$. This shows that $L=-(-\Delta)^{\alpha(x)}$ is a stable-like operator.  If $\alpha$ satisfies a H\"{o}lder condition, \cite{Bass-1988} shows that $L$ generates a unique ``stable-like'' Markov process (in dimension $d=1$), the multivariate case is discussed by \cite{ Hoh-2000, Negoro-1994}
    if $\alpha(x)$ is smooth and, recently, in \cite{Kuhn-thesis} for $\alpha(x)$ satisfying a H\"{o}lder condition. Note that a stable-like process is, in general, non-symmetric.  If $\alpha(x)\equiv\alpha$ is constant, then $L$ generates a rotationally invariant (hence, symmetric) $\alpha$-stable L\'evy process.

    Assume that $0<\underline{\alpha}\leq\alpha(x)\leq\alpha(x)\leq\overline{\alpha}<2$ for all $x\in\real^{d}$, and
    $$
        \int_0^{1}\frac{(\beta(u)|\log u|)^{2}}{u^{1+\overline{\alpha}}}\,du<\infty,
    $$
    where $\beta(u):=\sup_{|x-y|\leq u}|\alpha(x)-\alpha(y)|$. The (non-symmetric) kernel
    $$
        k(x,y):=\gamma(x)|y-x|^{-\alpha(x)-d},\quad x,y\in\real^{d},
    $$
    satisfies \eqref{c1.h} and defines a regular lower bounded semi-Dirichlet form on $L^{2}(\real^{d},dx)$; we call the corresponding Hunt process a \emph{stable-like process}. For $0<p\leq1$, define the corresponding family of conductances $C^{n,p}$, $n\in\nat$, by \eqref{eq1.5}. Clearly, for any $p\in(0,1]$ such that $1/p\geq\overline{\alpha}$, the conductances $C^{n,p}$, $n\in\nat$, satisfy \eqref{t1}--\eqref{t6}. Thus, the family of corresponding Markov chains $\{X^{n}_t\}_{t\geq0}$, $n\in\nat$, is tight. If there exist open balls $B_1,B_2\subseteq\real^{d}$ such that
    $$
        \inf_{(x,y)\in B_1\times B_2}|\alpha(x)-\alpha(y)|>0
    $$
    (this assumption implies \eqref{c2.h}), then it is easy to see with the above assumptions that \eqref{c2.h}--\eqref{c6.h} hold true (for the subsequence $n_i=2^{i}$, $i\in\nat$). It is also very easy to verify the conditions \eqref{c7.h} (i), (ii), \eqref{eq-tm1.10} and \eqref{c9.h} (in the context of Corollary \ref{c1.10}).  On the other hand,  conditions  \eqref{c8.h} and \eqref{c10.h} (again in the context of Corollary \ref{c1.10}) follow directly from  the dominated convergence theorem and Lebesgue's differentiation theorem in order to show $\lim_{n\to\infty}n^{d}C^{n,p}([x]_n,[x]_n+[y]_n)=\gamma(x)|y|^{-\alpha(x)-d}$.  Indeed, for $n\in\nat$, $p\in(0,1]$, $1/p\geq\overline{\alpha}$, $x\in\real^{d}$, $y\in B_{1}(0)$, $|[y]_n|>2\sqrt{d}/n^{p}$, we have
\begin{align*}
    n^{d}&C^{n,p}([x]_n,[x]_n+[y]_n)\\
    &= n^{2d}\int_{\bar{[x]}_n}\int_{[x]_n+\bar{[y]}_n} \frac{\gamma(u)\,dv\,du}{|v-u|^{\alpha(u)+d}}\\
    &= n^{2d}\int_{\bar{[x]}_n}\int_{[x]_n-u+\bar{[y]}_n} \frac{\gamma(u)\,dv\,du}{|v|^{\alpha(u)+d}}\\
    &\leq n^{2d}\,\overline{\gamma}\int_{\bar{[x]}_n}\int_{B_{\sqrt{d}/n}([y]_n)\cap B_1(0)} \frac{dv\, du}{|v|^{\alpha(u)+d}}
    + n^{2d}\,\overline{\gamma}\int_{\bar{[x]}_n}\int_{B_{\sqrt{d}/n}([y]_n)\cap B^{c}_1(0)} \frac{dv\,du}{|v|^{\alpha(u)+d}}\\
    &\leq n^{d}\overline{\gamma} \int_{B_{\sqrt{d}/n}([y]_n)} \frac{dv}{|v|^{\overline{\alpha}+d}}
    + n^{d}\overline{\gamma}\int_{B_{\sqrt{d}/n}([y]_n)} \frac{dv}{|v|^{\underline{\alpha}+d}}\\
    &= \overline{\gamma}\,2^{\overline{\alpha}+d}d^{d/2}\,V_d |y|^{-\overline{\alpha}-d}
    +\overline{\gamma}\,2^{\underline{\alpha}+d}d^{d/2}\,V_d |y|^{-\underline{\alpha}-d}\\
    &\leq\overline{\gamma}\,2^{\overline{\alpha}+d+1}d^{d/2}\,V_d |y|^{-\overline{\alpha}-d},
\end{align*}
where $\overline{\gamma}:=\sup_{x\in\real^{d}}\gamma(x)$  and $V_d$ denotes the volume of the $d$-dimensional unit ball. Similarly, for $n\in\nat$, $p\in(0,1]$, $1/p\geq\overline{\alpha}$, $x\in\real^{d}$, $y\in B^{c}_{1}(0)$, $|[y]_n|>2\sqrt{d}/n^{p}$, we have
$$
    n^{d}C^{n,p}([x]_n,[x]_n+[y]_n)
    \leq\overline{\gamma}\,2^{\overline{\alpha}+d+1}d^{d/2}\,V_d|y|^{-\underline{\alpha}-d}.
$$
The assertion now follows by an application of the dominated convergence  and Lebesgue's differentiation theorems.
Finally, let us verify \eqref{c11.h} (in the context of Corollary \ref{c1.10}). We proceed as follows
\begin{align*}
    &\int_{B_1(0)}|y| \big|k(x,x+y)-k(x,x-y)-n^{d}C^{n,p}([x]_n,[x]_n+[y]_n)+n^{d}C^{n,p}([x]_n,[x]_n-[y]_n)\big| \,dy\\
    &= n^{d}\int_{B_1(0)\setminus B_{\sqrt{d}/n^{p}}(0)} |y|
    \bigg|\int_{\bar{[x]}_n}\int_{\bar{[y]}_n+[x]_n}k(u,v)\,dv\,du
    - \int_{\bar{[x]}_n}\int_{-\bar{[y]}_n+[x]_n}k(u,v)\,dv\,du\bigg|\,dy\\
    &= n^{d}\int_{B_1(0)\setminus B_{\sqrt{d}/n^{p}}(0)}|y|
    \bigg|\int_{\bar{[x]}_n}\int_{\bar{[y]}_n+[x]_n-u}k(u,u+v)\,dv\,du\\
    &\hphantom{n^{d}\int_{B_1(0)\setminus B_{\sqrt{d}/n^{p}}(0)}|y|\bigg|\int_{\bar{[x]}_n}\int_{\bar{[y]}_n+[x]_n-u}}
    - \int_{\bar{[x]}_n}\int_{-\bar{[y]}_n+[x]_n-u} k(u,u+v)\,dv\,du\bigg|\,dy\\
    &= n^{d}\int_{B_1(0)\setminus B_{\sqrt{d}/n^{p}}(0)}|y|
    \int_{\bar{[x]}_n}\int_{\bar{[y]}_n}
    \big|\gamma(u) \big|v+[x]_n-u\big|^{-d-\alpha(u)}\\
    &\hphantom{n^{d}\int_{B_1(0)\setminus B_{\sqrt{d}/n^{p}}(0)}|y|
    \int_{\bar{[x]}_n}\int_{\bar{[y]}_n}
    \big|\gamma(u) \big|v++}
    - \gamma(u)\big|v-[x]_n+u\big|^{-d-\alpha(u)}\big|\,dv\,du\,dy\\
    &\leq 2(d+\underline{\alpha})\,\overline{\gamma}\,n^{d} \int_{B_1(0)\setminus B_{1-\sqrt{d}/2n}(0)}
    |y|\big( |[y]_n|-\sqrt{d}/2n \big)^{-\underline{\alpha}-d-1}\,dy \int_{\bar{[x]}_n}|[x]_n-u|\,du\\
    &\qquad\mbox{} + 2 (d+\overline{\alpha})\,\overline{\gamma}\,n^{d}
    \int_{B_{1-\sqrt{d}/2n}(0)\setminus B_{\sqrt{d}/n^{p}}(0)}|y| \big(|[y]_n|-\sqrt{d}/2n\big)^{-\overline{\alpha}-d-1}\,dy
    \int_{\bar{[x]}_n}|[x]_n-u|\,du\\
    &\leq 2^{\underline{\alpha}+d+2}\,(d+\overline{\alpha})\,\overline{\gamma}\,n^{d}
    \int_{B_1(0)\setminus B_{1-\sqrt{d}/2n}(0)}|y|^{-\underline{\alpha}-d}\,dy \int_{B_{\sqrt{d}/2n}(0)}|u|\,du\\
    &\qquad\mbox{}+ 2^{\overline{\alpha}+d+2}\,(d+\overline{\alpha})\,\overline{\gamma}\,n^{d}
    \int_{B_{1-\sqrt{d}/2n}(0)\setminus B_{\sqrt{d}/n^{p}}(0)}|y|^{-\overline{\alpha}-d}\,dy \int_{B_{\sqrt{d}/2n}(0)}|u|\,du\\
    &= \frac{2^{\underline{\alpha}+1}\, d^{(d+5)/2}\, (d+\overline{\alpha})\, \overline{\gamma}\,V^{2}_d}{\underline{\alpha}\,(d+1)\,n} \left[\bigg(\bigg(1-\frac{\sqrt{d}}{2n}\bigg)^{-\underline{\alpha}}-1\bigg)
    + \bigg(\bigg(\frac{\sqrt{d}}{n^{p}}\bigg)^{-\overline{\alpha}}-\bigg(1-\frac{\sqrt{d}}{2n}\bigg)^{-\overline{\alpha}}\bigg)\right],
%
\end{align*}
where we use Taylor's theorem in the fourth step. Thus, if $1/p>\overline{\alpha}$, \eqref{c11.h} follows directly from the dominated convergence theorem.
\end{example}

\section{Semimartingale approach}\label{s3}

 We can improve the convergence results of the previous section if we know that the limiting process is a semimartingale.    Let $\process{X}$ be the canonical process on the Skorokhod space $\mathbb{D}(\real^{d})$  and set $\mathcal{D}(\real^{d}):=\sigma\{X_t: t\geq0\}$; it is known that $\mathcal{D}(\real^{d})=\mathcal{B}(\mathbb{D}(\real^{d}))$, see \cite[Chapter VI]{Jacod-Shiryaev-2003}. The canonical filtration on the measurable space $(\mathbb{D}(\real^{d}),\mathcal{D}(\real^{d}))$ is given by $\textbf{D}(\real^{d})=\{\mathcal{D}_t(\real^{d})\}_{t\geq0}$,
 $\mathcal{D}_t(\real^{d}):=\bigcap_{s>t}\sigma\{X_u:u\leq s\}$.

Let $b:\real^{d}\to\real^{d}$ be a Borel function and $\nu:\real^{d}\times\mathcal{B}(\real^{d}\setminus\left\lbrace 0\right\rbrace )\to[0,\infty)$ a Borel kernel satisfying   $\int_{\real^{d}\setminus\left\lbrace 0\right\rbrace }(1\wedge|y|^{2})\nu(x,dy)<\infty$ for every $x\in\real^{d}$. Fix a truncation function $h:\real^{d}\to\real^{d}$ (see Section~\ref{sec21} for the definition) and define
\begin{align*}
    B_t                 &:= \int_0^t b(X_s)\,ds,\\
    \tilde{A}^{n,ik}_t  &:= \int_0^t\int_{\real^{d}} h_{i}(y)h_{k}(y)\,\nu(X_s,dy)\,ds, \quad i,k=1,\dots,d,\\
    N(ds,dy)            &:= \nu(X_s,dy)\,ds.
\end{align*}
Finally, assume
\begin{gather}
\label{c1.s}\tag{\textbf{C1.S}}
\parbox[t]{.88\linewidth}{for each $x\in\real^{d}$ there is a unique probability $\mathbb{P}_x(\bullet)$ on $(\mathbb{D}(\real^{d}),\mathcal{D}(\real^{d}))$ such that
    \begin{itemize}
    \item[(i)] $x\mapsto \mathbb{P}_x(B)$ is Borel measurable for all $B\in\mathcal{D}(\real^{d})$;
    \item [(ii)] $\mathbb{P}_x(X_0=x)=1$;
    \item [(iii)] $\process{X}$ is a semimartingale on the stochastic basis $(\mathbb{D}(\real^{d}),\mathcal{D}(\real^{d}),\mathbb{P}_x,\textbf{D}(\real^{d}))$ with modified characteristics $(B,\tilde{A},N)$.
    \end{itemize}
}
\end{gather}
Note that we assume that $X$ is a pure jump semimartingale, i.e.\ the continuous part -- the characteristic $A$ -- vanishes.

\begin{remark}
    Condition \eqref{c1.s} implies that $(\mathbb{D}(\real^{d}),\mathcal{D}(\real^{d}),\{\mathbb{P}_x\}_{x\in\real^{d}}, \textbf{D}(\real^{d}),\process{X})$ is a Markov process with an extended generator, defined for all $f\in C_b^{2}(\real^{d})$ by	
    $$
        \mathcal{A}f(x)=\sum_{i=1}^{d}b_{i}(x)\frac{\partial f(x)}{\partial x_{i}}
        +\int_{\real^{d}}\bigg(f(y+x)-f(x)-\sum_{i=1}^{d}h_{i}(x)\frac{\partial f(x)}{\partial x_{i}}\bigg)\nu(x,dy),
	$$
    see \cite[Remark IX.4.5]{Jacod-Shiryaev-2003}. For L\'evy processes (i.e. for the situation when $b(x)$ and $\nu(x,dy)$ do not depend on $x$), \eqref{c1.s} is trivially satisfied. If
    $$
  	\lim_{r\to\infty}\sup_{x\in\real^{d}}\nu(x,B^{c}_r(0))=0,\quad \sup_{x\in\real^{d}}|b(x)|<\infty,\quad
    	\sup_{x\in\real^{d}}\int_{\real^{d}}(1\wedge|y|^{2})\nu(x,dy)<\infty
    $$
    and \eqref{c2.s} (see below) hold, then there is at least one semimartingale $\process{X}$ with the given characteristics $(B,0,N)$,
    cf.\ \cite[Theorem IX.2.31]{Jacod-Shiryaev-2003}.

Conditions  ensuring  uniqueness of the family $\mathbb{P}_x(\cdot)$, $x\in\real^{d}$, are given in \cite[Theorem III.2.32]{Jacod-Shiryaev-2003}. In the context of Feller (or L\'evy-type) processes a different approach to existence and uniqueness of $\process{X}$ can be found in \cite[Theorem 4.6.7]{Jacob-Book-III-2005} and \cite[Chapter 3]{Bottcher-Schilling-Wang-2013}.
\end{remark}

\subsection{Convergence}\label{sce31}

Let $C^{n}:\ZZ_n^{d}\times\ZZ_n^{d}\to[0,\infty)$, $n\in\nat$, be a family of kernels satisfying \eqref{t1} and \eqref{t2}. As we have already seen in Section \ref{s2}, these kernels define a family of regular Markov semimartingales $\process{X^{n}}$, $n\in\nat$, on $\ZZ_n^{d}$.  If we apply the results of \cite[Theorem IX.4.8]{Jacod-Shiryaev-2003} in our setting, we arrive at the following theorem.
\begin{theorem}\label{tm2.1}
    Let $C^{n}:\ZZ_n^{d}\times\ZZ_n^{d}\to[0,\infty)$ be as before, and assume that \eqref{c1.s} holds as well as
    \begin{gather}
    \label{c2.s}\tag{\textbf{C2.S}}
        \parbox[t]{.88\linewidth}{the functions $b(x)$, $x\mapsto\int_{\real^{d}}h_{i}(y)h_{k}(y)\nu(x,dy)$ and $x\mapsto\int_{\real^{d}}g(y)\nu(x,dy)$ are continuous for all $i,k=1,\dots,d$ and all bounded and continuous functions $g:\real^{d}\to\real$ vanishing in a neighbourhood of the origin};\\
    \label{c3.s}\tag{\textbf{C3.S}}
        \parbox[t]{.88\linewidth}{$\displaystyle\forall R>0\::\:
        \lim_{r\to\infty}\sup_{x\in B_R(0)}\nu(x,B^{c}_r(0))=0;$}\\
    \label{c4.s}\tag{\textbf{C4.S}}
        \parbox[t]{.88\linewidth}{$\forall R>0,\: i=1,\dots, d\::\:$
        $$
            \lim_{n\to\infty}\sup_{x\in B_R(0)}\bigg|\sum_{b\in\ZZ_n^{d}}h_{i}(b)C^{n}([x]_n,[x]_n+b)-b_{i}(x)\bigg|=0;
        $$}\\
    \label{c5.s}\tag{\textbf{C5.S}}
        \parbox[t]{.88\linewidth}{$\forall R>0,\: i=1,\dots, d\::$}\\
    \notag
        \qquad\quad\lim_{n\to\infty}\sup_{x\in B_R(0)} \bigg|\sum_{b\in\ZZ_n^{d}}h_{i}(b)h_{k}(b)C^{n}([x]_n,[x]_n+b)-\int_{\real^{d}}h_{i}(y)h_{k}(y)\,\nu(x,dy)\bigg|=0;\\
    \label{c6.s}\tag{\textbf{C6.S}}
        \parbox[t]{.88\linewidth}{$\forall R>0,\: \forall f\in C_b(\real^d,\real),\text{vanishing in a neighbourhood of $0$:}$
        $$
            \lim_{n\to\infty}\sup_{x\in B_R(0)}\bigg|\sum_{b\in\ZZ_n^{d}}g(b)C^{n}([x]_n,[x]_n+b)-\int_{\real^{d}}g(y)\,\nu(x,dy)\bigg|=0.
        $$}
    \end{gather}
    If the initial distributions of $\process{X^{n}}$, $n\in\nat$, converge weakly to that of $\process{X}$, then
    $\left\{X^{n}_t\right\}_{t\geq0}\xrightarrow[n\to\infty]{d}\process{X}$ in Skorokhod space.
\end{theorem}

\subsection{Approximation}\label{sec32}

Using the same notation  as in Section \ref{sec22}, for $0<p\leq1$, we define a family of  kernels  $C^{n,p}:\ZZ_n^{d}\times\ZZ_n^{d}\to[0,\infty)$, $n\in\nat$, by
\begin{equation}\label{eq2.1}C^{n,p}(a,b)
:= \begin{cases}
    \nu(a,\bar{b}-a), & |a-b|>\frac{\sqrt{d}}{n^{p}} \\
    0, & |a-b|\leq \frac{\sqrt{d}}{n^{p}}.
    \end{cases}
\end{equation}
The kernles $C^{n,p}$, $n\in\nat$, automatically satisfy \eqref{t1}, the conditions \eqref{t2} and \eqref{t3} are ensured by
\begin{gather}\label{t1.s}\tag{\textbf{T1.S}}
    \forall \rho>0\::\:
    \sup_{x\in\real^{d}}\nu(x,B^{c}_\rho(0))<\infty.
\end{gather}
Hence, under \eqref{t1.s}, there is a family of regular Markov semimartingales  $\process{X^{n}}$, $n\in\nat$, with  (modified) characteristics of the form
\begin{align*}
    B^{n}(h)_t
    &=\int_0^{t}\sum_{a\in\ZZ_n^{d}}h(a)\,\nu(X_s^{n},\bar{a})\,ds,\\
    A^{n}_t&=0,\\
    \tilde{A}^{n}(h)_t^{ik}&=\int_0^{t}\sum_{a\in\ZZ_n^{d}}h_{i}(a)h_{k}(a)\,\nu(X_s^{n},\bar{a})\,ds,\\
    N^{n}(ds,b)&=\nu(X_s^{n},\bar{a})\,ds
\end{align*}
with some fixed truncation function $h:\real^{d}\to\real^{d}$.

Let us give sufficient conditions for \eqref{c4.s}--\eqref{c6.s}.

\begin{proposition}\label{p2.2}
\textup{(i)} The condition \eqref{c4.s} will be satisfied if one of the following two conditions holds:
\begin{gather}\label{c4.s1}\tag{\textbf{C4.S.1}}
    \parbox[t]{.85\linewidth}{
    there exists $\rho>0$ such that $\nu(x,dy)$ is symmetric on $B_\rho(0)$ for all $x\in\real^{d}$ and\newline
    \phantom{MM}$
    \begin{aligned}
    &\sup_{x\in B_R(0)}\nu(x,B^{c}_r(0))<\infty,\quad r>0,\\
    &\lim_{r\to\infty}\sup_{x\in B_R(0)}\nu(x,B^{c}_r(0))=0,\\
    &\lim_{n\to\infty}\sup_{x\in B_R(0)}  \lVert\nu([x]_n,\bullet)-\nu(x,\bullet)\rVert_{\mathrm{TV}(B^{c}_\varepsilon(0))}=0,\quad \varepsilon>0,\\
    &b_i(x)=\lim_{\varepsilon\downto 0}\int_{B^{c}_\varepsilon(0)}h_i(y)\,\nu(x,dy)
    \end{aligned}$\newline
    hold  for all  $R>0$ and $i=1,\dots,d$, where $\lVert\mu(\bullet)\rVert_{\mathrm{TV}(A)}$  denotes the  total variation of the signed measure $\mu(\bullet\cap A)$;}\\
\label{c4.s2}\tag{\textbf{C4.S.2}}
    \parbox[t]{.85\linewidth}{
    \phantom{MM}$
    \begin{aligned}
    &\sup_{x\in B_R(0)}\nu(x,B^{c}_r(0))<\infty,\quad\text{for all\ } r>0,\\
    &\lim_{r\to\infty}\sup_{x\in B_R(0)}\nu(x,B^{c}_r(0))=0,\\
    &\lim_{\varepsilon\downto  0}\varepsilon\sup_{x\in B_R(0)}\nu(x,B_1(0)\setminus  B_{\varepsilon^{p}}(0))=0,\\
    &\lim_{\varepsilon\downto 0}\varepsilon^{p}\sup_{x\in B_R(0)}\nu\big(x,B_{\sqrt{d}\varepsilon^{p}+(\sqrt{d}/2)\varepsilon}(0)\setminus B_{\sqrt{d}\varepsilon^{p}-(\sqrt{d}/2)\varepsilon}(0)\big) = 0,\\
    &\lim_{n\to\infty}\sup_{x\in B_R(0)}\int_{B_{1}(0)\setminus B_{\sqrt{d}/n^{p}-\sqrt{d}/2n}(0)}|y|\lVert\nu([x]_n,dy)-\nu(x,dy)\rVert_{\mathrm{TV}}=0,\\
    &\lim_{n\to\infty}\sup_{x\in B_R(0)}  \lVert\nu([x]_n,\bullet)-\nu(x,\bullet)\rVert_{\mathrm{TV}(B^{c}_\varepsilon(0))}=0,\quad \varepsilon>0,\\
    &\lim_{\varepsilon\downto 0}\sup_{x\in B_R(0)}\left|\int_{B^{c}_\varepsilon(0)}h_i(y)\,\nu(x,dy)-b_i(x)\right|=0
    \end{aligned}$\newline
    hold  for all  $R>0$ and $i=1,\dots,d$.}
\end{gather}

\noindent
\textup{(ii)} The condition \eqref{c5.s} will be satisfied if
\begin{gather}
\label{c5.s1}\tag{\textbf{C5.S.1}}
    \parbox[t]{.85\linewidth}{\phantom{MM}$
    \begin{aligned}
    &\sup_{x\in B_R(0)}\nu(x,B^{c}_r(0))<\infty,\quad r>0,\\
    &\lim_{r\to\infty}\sup_{x\in B_R(0)}\nu(x,B^{c}_r(0))=0,\\
    &\lim_{\varepsilon\downto 0}\varepsilon\sup_{x\in B_R(0)}\int_{B_1(0)\setminus B_{\varepsilon^{p}}(0)}|y|\,\nu(x,dy)=0,\\
    &\lim_{\varepsilon\downto 0}\sup_{x\in B_R(0)}\int_{B_\varepsilon(0)}|y|^{2}\,\nu(x,dy)=0,\\
    &\lim_{n\to\infty}\sup_{x\in B_R(0)}\int_{B_1(0)}|y|^{2}\lVert\nu([x]_n,dy)-\nu(x,dy)\rVert_{\mathrm{TV}}=0,\\
    &\lim_{n\to\infty}\sup_{x\in B_R(0)}  \lVert\nu([x]_n,\bullet)-\nu(x,\bullet)\rVert_{\mathrm{TV}(B^{c}_\varepsilon(0))}=0,\quad \varepsilon>0,
    \end{aligned}$\newline
    hold  for all  $R>0$.}
\end{gather}

\noindent
\textup{(iii)} The condition \eqref{c6.s} will be satisfied if
\begin{gather}
\label{c6.s1}\tag{\textbf{C6.S.1}}
    \parbox[t]{.85\linewidth}{for any $R>0$,\newline
    \mbox{}$\qquad\displaystyle
    \begin{aligned}
    &\sup_{x\in B_R(0)}\nu(x,B^{c}_r(0))<\infty,\quad r>0,\\
    &\lim_{r\to\infty}\sup_{x\in B_R(0)}\nu(x,B^{c}_r(0))=0,\\
    &\lim_{n\to\infty}\sup_{x\in B_R(0)}  \lVert\nu([x]_n,\bullet)-\nu(x,\bullet)\rVert_{\mathrm{TV}(B^{c}_\varepsilon(0))}=0,\quad \varepsilon>0.
    \end{aligned}$}
\end{gather}

\end{proposition}
\begin{proof}
    Let us first check \eqref{c4.s}. For every $i\in\{1,\dots,d\}$ we have
    \begin{align*}
    \bigg|&\sum_{b\in\ZZ_n^{d}}h_{i}(b)C^{n,p}([x]_n,[x]_n+b)-b_{i}(x)\bigg|
    =\bigg|\sum_{|b|>\sqrt{d}/n^{p}}h_{i}(b)\nu([x]_n,\bar{b})-b_{i}(x)\bigg|\\
    &\leq\bigg|\sum_{|b|>\sqrt{d}/n^{p}}h_{i}(b)\big(\nu([x]_n,\bar{b})-\nu(x,\bar{b})\big)\bigg| +\bigg|\sum_{|b|>\sqrt{d}/n^{p}}h_{i}(b)\nu(x,\bar{b})-b_{i}(x)\bigg|\\
    &\leq  \bigg|\sum_{|b|>\sqrt{d}/n^{p}}\int\limits_{\bar{b}} \big(h_{i}(b)-h_i(y)\big) \big(\nu([x]_n,dy)-\nu(x,dy)\big)\bigg|\\
    &\quad\mbox{}+ \bigg|\sum_{|b|>\sqrt{d}/n^{p}}\int_{\bar{b}}h_i(y) \big(\nu([x]_n,dy)-\nu(x,dy)\big) -\int\limits_{B^{c}_{\sqrt{d}/n^{p}-\sqrt{d}/2n}(0)}\!\!\!\!\!\!h_i(y)\big(\nu([x]_n,dy)-\nu(x,dy)\big)\bigg|\\
    &\quad\mbox{}+\bigg|\int\limits_{B^{c}_{\sqrt{d}/n^{p}-\sqrt{d}/2n}(0)}\!\!\! h_i(y)\big(\nu([x]_n,dy)-\nu(x,dy)\big)\bigg| +\bigg\rvert\sum_{|b|>\sqrt{d}/n^{p}}\int\limits_{\bar{b}}\big(h_i(b)-h_i(y)\big)\,\nu(x,dy)\bigg|\\
    &\quad\mbox{}+\bigg|\sum_{|b|>\sqrt{d}/n^{p}}\int_{\bar{b}}h_i(y)\,\nu(x,dy) -\int_{B^{c}_{\sqrt{d}/n^{p}-\sqrt{d}/2n}(0)}h_i(y)\,\nu(x,dy)\bigg\rvert\\
    &\quad\mbox{}+\bigg|\int_{B^{c}_{\sqrt{d}/n^{p}-\sqrt{d}/2n}(0)}h_i(y)\,\nu(x,dy)-b_{i}(x)\bigg\rvert.
\end{align*}
Pick $0<\eta<\rho\wedge1$ such that $h(y)=y$ for all $y\in B_{\eta}(0)$ (recall that $\rho>0$ appears in \eqref{c4.s1}). For all $n\in\nat$, $n^{p}>3\sqrt{d}/2\eta$, we have that
\begin{align*}
    &\bigg|\sum_{b\in\ZZ_n^{d}}h_{i}(b)C^{n,p}([x]_n,[x]_n+b)-b_{i}(x)\bigg|\\
    &\leq  \bigg|\sum_{\sqrt{d}/n^{p}<|b|<\eta-\sqrt{d}/2n}\int_{\bar{b}}(b_{i}-y_i) \big(\nu([x]_n,dy)-\nu(x,dy)\big)\bigg|\\
    &\quad\mbox{}+ \bigg|\sum_{|b|\geq\eta-\sqrt{d}/2n}\int_{\bar{b}}\big(h_{i}(b)-h_i(y)\big) \big(\nu([x]_n,dy)-\nu(x,dy)\big)\bigg|\\
    &\quad\mbox{}+ \bigg|\int_{B_{\sqrt{d}/n^{p}+\sqrt{d}/2n}(0)\setminus B_{\sqrt{d}/n^{p}-\sqrt{d}/2n}(0)}y_i\,\big(\nu([x]_n,dy)-\nu(x,dy)\big)\bigg|\\
    &\quad\mbox{}+\bigg|\int_{B_\eta(0)\setminus B_{\sqrt{d}/n^{p}-\sqrt{d}/2n}(0)}\!\!\!y_i\,\big(\nu([x]_n,dy)-\nu(x,dy)\big)\bigg|+
    \lVert h\rVert_\infty  \lVert\nu([x]_n,\bullet)-\nu(x,\bullet)\rVert_{\mathrm{TV}(B^{c}_{\eta}(0))}\\
    &\quad\mbox{}+\bigg\rvert\sum_{\sqrt{d}/n^{p}<|b|<\eta-\sqrt{d}/2n}\int_{\bar{b}}(b_i-y_i)\,\nu(x,dy)\bigg|
    +\bigg\rvert\sum_{|b|\geq\eta-\sqrt{d}/2n}\int_{\bar{b}}\big(h_i(b)-h_i(y)\big)\,\nu(x,dy)\bigg|\\
    &\quad\mbox{}+\bigg|\int_{B_{\sqrt{d}/n^{p}+\sqrt{d}/2n}(0)\setminus B_{\sqrt{d}/n^{p}-\sqrt{d}/2n}(0)} y_i\,\nu(x,dy)\bigg\rvert +\bigg|\int_{B^{c}_{\sqrt{d}/n^{p}-\sqrt{d}/2n}(0)}h_i(y)\,\nu(x,dy)-b_{i}(x)\bigg\rvert.
\end{align*}
Fix $R>0$ and $\varepsilon>0$. According to \eqref{c4.s1} there is some $r_0>0$ such that
$$
    \sup_{x\in B_R(0)}\nu(x,B^{c}_r(0))<\varepsilon,\quad r\geq r_0.
$$
Since $h_i(y)$ is uniformly continuous on the compact set $\bar{B}_{2r_0}(0))$, 
there is some $\delta>0$ such that $|h_i(y_1)-h_i(y_2)|<\varepsilon$ for all $y_1,y_2\in \bar{B}_{2r_0}(0)$, $|y_1-y_2|<\delta$.
Now, for $n\in\nat$, $n^{p}>3\sqrt{d}/2\eta\vee \sqrt{d}/2\delta$, we have that
\begin{align*}
    &\bigg|\sum_{b\in\ZZ_n^{d}}h_{i}(b)C^{n,p}([x]_n,[x]_n+b)-b_{i}(x)\bigg|\\
    &\leq  \bigg|\sum_{\sqrt{d}/n^{p}<|b|<\eta-\sqrt{d}/2n}\int_{\bar{b}}(b_{i}-y_i) \big(\nu([x]_n,dy)-\nu(x,dy)\big)\bigg|\\
    &\quad\mbox{}+ \varepsilon \lVert\nu([x]_n,\bullet)-\nu(x,\bullet)\rVert_{\mathrm{TV}(B_{2r_0}(0)\setminus B_{\eta/2}(0))}
    +4\varepsilon\lVert h\rVert_\infty\\
    &\quad\mbox{}+ \bigg|\int_{B_{\sqrt{d}/n^{p}+\sqrt{d}/2n}(0)\setminus B_{\sqrt{d}/n^{p}-\sqrt{d}/2n}(0)}y_i \big(\nu([x]_n,dy)-\nu(x,dy)\big)\bigg|\\
    &\quad\mbox{}+\bigg|\int_{B_\eta(0)\setminus B_{\sqrt{d}/n^{p}-\sqrt{d}/2n}(0)}y_i \big(\nu([x]_n,dy)-\nu(x,dy)\big)\bigg|+
    \lVert h\rVert_\infty  \lVert\nu([x]_n,\bullet)-\nu(x,\bullet)\rVert_{\mathrm{TV}(B^{c}_{\eta}(0))}\\
    &\quad\mbox{}+\bigg\rvert\sum_{\sqrt{d}/n^{p}<|b|<\eta-\sqrt{d}/2n}\int_{\bar{b}}(b_i-y_i)\,\nu(x,dy)\bigg|
    +\varepsilon\nu\big(x,B_{2r_0}(0)\setminus B_{\eta/2}(0)\big)
    +2\varepsilon\lVert h\rVert_\infty\\
    &\quad\mbox{}+\bigg|\int_{B_{\sqrt{d}/n^{p}+\sqrt{d}/2n}(0)\setminus B_{\sqrt{d}/n^{p}-\sqrt{d}/2n}(0)} y_i\,\nu(x,dy)\bigg\rvert +\bigg|\int_{B^{c}_{\sqrt{d}/n^{p}-\sqrt{d}/2n}(0)}h_i(y)\,\nu(x,dy)-b_{i}(x)\bigg\rvert.
\end{align*}
This shows that \eqref{c4.s1} implies \eqref{c4.s}. Further, we have that
\begin{align*}
    &\bigg|\sum_{b\in\ZZ_n^{d}}h_{i}(b)C^{n,p}([x]_n,[x]_n+b)-b_{i}(x)\bigg|\\
    &\leq \frac{\sqrt{d}}{2n}\, \lVert\nu([x]_n,\bullet)-\nu(x,\bullet)\rVert_{\mathrm{TV}(B_{\eta}(0)\setminus B_{\sqrt{d}/n^{p}-\sqrt{d}/2n}(0))}\\
    &\quad\mbox{}+ \varepsilon  \lVert\nu([x]_n,\bullet)-\nu(x,\bullet)\rVert_{\mathrm{TV}(B_{2r_0}(0)\setminus B_{\eta/2}(0))} +4\varepsilon\lVert h\rVert_\infty\\
    &\quad\mbox{}+\int_{B_{\sqrt{d}/n^{p}+\sqrt{d}/2n}(0)\setminus B_{\sqrt{d}/n^{p}-\sqrt{d}/2n}(0)} |y|\, \lVert\nu([x]_n,dy)-\nu(x,dy)\rVert_{\mathrm{TV}}\\
    &\quad\mbox{}+ \int_{B_{\eta}(0)\setminus B_{\sqrt{d}/n^{p}-\sqrt{d}/2n}(0)} |y| \, \lVert\nu([x]_n,dy)-\nu(x,dy)\rVert_{\mathrm{TV}}\\
    &\quad\mbox{}+ \lVert h\rVert_\infty  \lVert\nu([x]_n,\bullet)-\nu(x,\bullet)\rVert_{\mathrm{TV}(B^{c}_{\eta}(0))}\\
    &\quad\mbox{}+\frac{\sqrt{d}}{2n}\,\nu\big(x,B_{\eta}(0)\setminus B_{\sqrt{d}/n^{p}-\sqrt{d}/2n}(0)\big)
    + \varepsilon\nu\big(x,B_{2r_0}(0)\setminus B_{\eta/2}(0)\big)
    + 2\varepsilon\lVert h\rVert_\infty\\
    &\quad\mbox{}+\frac{3\sqrt{d}}{2n^{p}}\, \nu\big(x,B_{\sqrt{d}/n^{p}+\sqrt{d}/2n}(0)\setminus B_{\sqrt{d}/n^{p}-\sqrt{d}/2n}(0)\big) +\bigg|\int_{B^{c}_{\sqrt{d}/n^{p}-\sqrt{d}/2n}(0)}h_i(y)\,\nu(x,dy)-b_{i}(x)\bigg\rvert.
\end{align*}
Thus, \eqref{c4.s2} implies \eqref{c4.s}, too.

Now we turn to \eqref{c5.s}. Fix $i,k\in\{1,\dots,d\}$. We have
\begin{align*}
    &\bigg|\sum_{b\in\ZZ_n^{d}}h_{i}(b)h_{k}(b)C^{n}([x]_n,[x]_n+b)-\int_{\real^{d}}h_{i}(y)h_{k}(y)\,\nu(x,dy)\bigg|\\
    &=\bigg|\sum_{|b|>\sqrt{d}/n^{p}}h_{i}(b)h_{k}(b)\nu([x]_n,\bar{b})-\int_{\real^{d}}h_{i}(y)h_{k}(y)\,\nu(x,dy)\bigg|\\
    &\leq\bigg|\sum_{|b|>\sqrt{d}/n^{p}}h_{i}(b)h_{k}(b)\nu([x]_n,\bar{b})-\sum_{|b|>\sqrt{d}/n^{p}}h_{i}(b)h_{k}(b)\nu(x,\bar{b})\bigg|\\
    &\qquad\mbox{}+\bigg|\sum_{|b|>\sqrt{d}/n^{p}}\int_{\bar{b}}\big(h_{i}(b)h_{k}(b)-h_i(y)h_k(y)\big)\,\nu(x,dy)\bigg|
    +\bigg|\sum_{|b|\leq\sqrt{d}/n^{p}}\int_{\bar{b}}h_i(y)h_k(y)\,\nu(x,dy)\bigg|.
\end{align*}
Pick $0<\eta<1$ be such that $h(y)=y$ for all $y\in B_\eta(0)$. Thus, for all $n\in\nat$, $n^{p}>3\sqrt{d}/2\eta$,
\begin{align*}
    &\bigg|\sum_{b\in\ZZ_n^{d}}h_{i}(b)h_{k}(b)C^{n}([x]_n,[x]_n+b)-\int_{\real^{d}}h_{i}(y)h_{k}(y)\,\nu(x,dy)\bigg|\\
    &\leq \sum_{\sqrt{d}/n^{p}<|b|<\eta-\sqrt{d}/2n}|b_{i}||b_{k}| \bigg| \nu([x]_n,\bar{b})-\nu(x,\bar{b})\bigg|+2\lVert h\rVert^{2}_\infty\sum_{|b|\geq\eta-\sqrt{d}/2n} \bigg| \nu([x]_n,\bar{b})-\nu(x,\bar{b})\bigg|\\
    &\quad\mbox{}+\bigg|\sum_{\sqrt{d}/n^{p}<|b|<\eta-\sqrt{d}/2n}\int_{\bar{b}}(b_{i}b_{k}-y_iy_k)\,\nu(x,dy)\bigg|\\
    &\quad\mbox{}+\bigg|\sum_{|b|\geq\eta-\sqrt{d}/2n}\int_{\bar{b}}\big(h_{i}(b)h_{k}(b)-h_i(y)h_k(y)\big)\,\nu(x,dy)\bigg|\\
    &\quad\mbox{}+\int_{B_{\sqrt{d}/n^{p}+\sqrt{d}/2n}(0)} |y_i||y_k| \,\nu(x,dy) \\
    &\leq 4\int_{B_\eta(0)}|y|^{2} \lVert\nu([x]_n,dy)-\nu(x,dy)\rVert_{\mathrm{TV}}
    +2\lVert h\rVert^{2}_\infty  \lVert\nu([x]_n,\bullet)-\nu(x,\bullet)\rVert_{\mathrm{TV}(B^{c}_{\eta/2}(0))}\\
    &\quad\mbox{}+\frac{3}{2n}\int_{B_\eta(0)\setminus B_{\sqrt{d}/n^{p}-\sqrt{d}/2n}(0)} |y| \,\nu(x,dy)\\
    &\quad\mbox{}+\bigg|\sum_{|b|\geq\eta-\sqrt{d}/2n} \int_{\bar{b}} \big(h_{i}(b)h_{k}(b)-h_i(y)h_k(y)\big)\,\nu(x,dy)\bigg|\\
    &\quad\mbox{}+\int_{B_{\sqrt{d}/n^{p}+\sqrt{d}/2n}(0)} |y|^{2}\,\nu(x,dy).
\end{align*}
In the final step we use the elementary estimates
$|b_ib_k-y_iy_k|\leq (3/2n)|y|$ and $|b_{i}||b_{k}|\leq4|y|^{2}$ for all $y\in \bar{b}$.

Fix $R>0$ and $\varepsilon>0$. As before \eqref{c5.s1} shows that there are $r_0>0$ and $\delta>0$, such that
$$
    \sup_{x\in B_R(0)}\nu(x,B^{c}_r(0))<\varepsilon,\quad r\geq r_0,
$$
and $|h_i(y_1)h_k(y_1)-h_i(y_2)h_k(y_2)|<\varepsilon$ for all $y_1,y_2\in \bar{B}_{2r_0}(0)$, $|y_1-y_2|<\delta$. Consequently, for any $n\in\nat$, $n^{p}>3\sqrt{d}/2\eta\vee\sqrt{d}/2\delta$,
\begin{align*}
    &\bigg|\sum_{b\in\ZZ_n^{d}}h_{i}(b)h_{k}(b)C^{n}([x]_n,[x]_n+b)-\int_{\real^{d}}h_{i}(y)h_{k}(y)\,\nu(x,dy)\bigg|\\
    &\leq 4\int_{B_\eta(0)}|y|^{2} \, \lVert\nu([x]_n,dy)-\nu(x,dy)\rVert_{\mathrm{TV}}
    + 2\lVert h\rVert^{2}_\infty  \lVert\nu([x]_n,\bullet)-\nu(x,\bullet)\rVert_{\mathrm{TV}(B^{c}_{\eta/2}(0))}\\
    &\qquad\mbox{}+ \frac{2\sqrt{d}}{n}\int_{B_\eta(0)\setminus B_{\sqrt{d}/n^{p}-\sqrt{d}/2n}(0)} |y|\,\nu(x,dy) +\varepsilon\nu\big(x,B_{2r_0}\setminus B_\eta(0)\big)
    +2\varepsilon\lVert h\rVert^{2}_\infty \\
    &\qquad\mbox{}+\int_{B_{\sqrt{d}/n^{p}+\sqrt{d}/2n}(0)} |y|^{2}\,\nu(x,dy).
\end{align*}
Therefore, \eqref{c5.s} is a direct consequence of \eqref{c5.s1}.

Let us, finally, discuss \eqref{c6.s}. Fix any bounded continuous function $g:\real^{d}\to\real$ which vanishes, for some $\eta>0$, on $B_\eta(0)$. For all $n\in\nat$, $n^{p}>\sqrt{d}/2\eta$, we have
\begin{align*}
    &\bigg|\sum_{b\in\ZZ_n^{d}}g(b)C^{n}([x]_n,[x]_n+b)-\int_{\real^{d}}g(y)\,\nu(x,dy)\bigg|\\
    &=\bigg|\sum_{|b|>\sqrt{d}/n^{d}}g(b)\nu([x]_n,\bar{b}) - \int_{\real^{d}}g(y)\,\nu(x,dy)\bigg|\\
    &\leq\bigg|\sum_{|b|>\sqrt{d}/n^{p}}g(b)\nu([x]_n,\bar{b}) - \sum_{|b|>\sqrt{d}/n^{p}}g(b)\nu(x,\bar{b})\bigg|\\
    &\qquad\mbox{}+ \bigg|\sum_{|b|>\sqrt{d}/n^{p}}g(b)\nu(x,\bar{b}) - \int_{\real^{d}}g(y)\,\nu(x,dy)\bigg|\\
    & \leq\|g\|_\infty  \lVert\nu([x]_n,\bullet)-\nu(x,\bullet)\rVert_{\mathrm{TV}(B^{c}_{\eta/2}(0))}
    + \bigg|\sum_{|b|>\eta-\sqrt{d}/2n} \int_{\bar{b}} \big(g(b)-g(y)\big) \,\nu(x,dy)\bigg|.
\end{align*}
Fix  $R>0$ and $\varepsilon>0$, and pick, as before, $r_0>0$ and $\delta>0$ such that
$$
    \sup_{x\in B_R(0)}\nu(x,B^{c}_r(0))<\varepsilon,\quad r\geq r_0
$$
and $|g(y_1)-g(y_2)|<\varepsilon$ for all $y_1,y_2\in \bar{B}_{2r_{0}}(0)$, $|y_1-y_2|<\delta$. Thus, for all $n\in\nat$, $n^{p}>\sqrt{d}/2\eta\vee\sqrt{d}/2\delta$,
\begin{align*}
    &\bigg|\sum_{b\in\ZZ_n^{d}}g(b)C^{n}([x]_n,[x]_n+b) - \int_{\real^{d}}g(y)\,\nu(x,dy)\bigg|\\
    &\qquad\leq\|g\|_\infty  \lVert\nu([x]_n,\bullet)-\nu(x,\bullet)\rVert_{\mathrm{TV}(B^{c}_{\eta/2}(0))}
    + \varepsilon\nu\big(x,B_{2r_0}(0)\setminus B_{\eta}(0)\big)
    + 2\varepsilon\|g\|_\infty,
\end{align*}
which, in view of \eqref{c6.s1}, concludes the proof.
\end{proof}

In the following  we discuss the situation when $\int_{\real^{d}}|y|^{2}\nu(x,dy)<\infty$ for all $x\in\real^{d}$. The following result is a direct consequence of \cite[Theorem IX.4.15]{Jacod-Shiryaev-2003}.
\begin{theorem}\label{tm2.3}
    Let $C^{n}:\ZZ_n^{d}\times\ZZ_n^{d}\to[0,\infty)$, $n\in\nat$, be a family of  kernels  satisfying \eqref{t1} and \eqref{t2}, and let $\process{X^{n}}$, $n\in\nat$, be the corresponding family of regular Markov semimartingales on $\ZZ_n^{d}$. Assume that $\int_{\real^{d}}|y|^{2}\nu(x,dy)<\infty$ for all $x\in\real^{d}$ and that the following conditions hold: \eqref{c1.s} and
    \begin{gather}
    \label{c2.sq}\tag{\textbf{C2.$\mathbf{S^{2}}$}}
    \parbox[t]{.88\linewidth}{
        the functions $x\mapsto b_{i}(x)+\int_{\real^{d}}(y_{i}-h_{i}(y))\nu(x,dy)$, $x\mapsto\int_{\real^{d}}y_{i}y_{k}\nu(x,dy)$, and $x\mapsto\int_{\real^{d}}g(y)\nu(x,dy)$ are continuous for all $i,k=1,\dots,d$ and all bounded and continuous functions $g:\real^{d}\to\real$ vanishing in a neighbourhood of the origin;}\\
    \label{c3.sq}\tag{\textbf{C3.$\mathbf{S^{2}}$}}
    \parbox[t]{.88\linewidth}{
        for all $R>0$, $\displaystyle\qquad\lim_{r\to\infty}\sup_{x\in B_R(0)}\int_{B^{c}_r(0)}|y|^{2}\,\nu(x,dy)=0;$}\\
    \label{c4.sq}\tag{\textbf{C4.$\mathbf{S^{2}}$}}
    \parbox[t]{.88\linewidth}{for all $R>0$ and all $i=1,\dots, d$,}\\
    \notag
        \lim_{n\to\infty}\sup_{x\in B_R(0)}\bigg|\sum_{b\in\ZZ_n^{d}}b_{i}C^{n}([x]_n,[x]_n+b)-b_{i}(x)-\int_{\real^{d}}(y_i-h_i(y))\,\nu(x,dy)\bigg|=0;\\
    \label{c5.sq}\tag{\textbf{C5.$\mathbf{S^{2}}$}}
    \parbox[t]{.88\linewidth}{
        for all $R>0$ and all $i=1,\dots,d$,}\\
    \notag
        \lim_{n\to\infty}\sup_{x\in B_R(0)} \bigg|\sum_{b\in\ZZ_n^{d}}b_{i}b_{k}C^{n}([x]_n,[x]_n+b)-\int_{\real^{d}}y_{i}y_{k}\,\nu(x,dy)\bigg|=0;\\
    \label{c6.sq}\tag{\textbf{C6.$\mathbf{S^{2}}$}}
    \parbox[t]{.88\linewidth}{
        for all $R>0$ and all  bounded and continuous functions $g:\real^{d}\to\real$ vanishing in a neighbourhood of the origin,}\\
    \notag
        \lim_{n\to\infty}\sup_{x\in B_R(0)}\bigg|\sum_{b\in\ZZ_n^{d}}g(b)C^{n}([x]_n,[x]_n+b)-\int_{\real^{d}}g(y)\,\nu(x,dy)\bigg|=0.
    \end{gather}
    If the initial distributions of  $\process{X^{n}}$, $n\in\nat$,  converge weakly to that of $\process{X}$, then
    $$
        \left\{X^{n}_t\right\}_{t\geq0}\xrightarrow[n\to\infty]{d}\process{X}
    $$
    in Skorokhod space.
\end{theorem}

If we use the discretisation \eqref{eq2.1} of $k(x,y)$, the proof of Proposition \ref{p2.2} applies and yields
\begin{proposition}\label{p2.4}
\textup{(i)} The condition \eqref{c4.sq} will be satisfied if either \eqref{c4.s1} or \eqref{c4.s2} and
    \begin{gather}
    \label{c2.sq1}\tag{\textbf{C2.$\mathbf{S^{2}}$.1}}
    \parbox[t]{.85\linewidth}{for all $R>0$,}\\
    \notag
        \lim_{n\to\infty}\sup_{x\in B_R(0)}\int_{B_{1}(0)}|y|\,\lVert\nu([x]_n,dy)-\nu(x,dy)\rVert_{\mathrm{TV}}=0
    \end{gather}
hold true.

\medskip
\noindent
\textup{(ii)} The condition \eqref{c5.sq} will be satisfied if
    \begin{gather}
    \label{c5.sq1}\tag{\textbf{C5.$\mathbf{S^{2}}$.1}}
    \parbox[t]{.85\linewidth}{for all $R>0$}\\
\notag
    \begin{aligned}[t]
    &\lim_{\varepsilon\downto 0}\varepsilon\sup_{x\in B_R(0)}\int_{B^{c}_{\varepsilon^{p}}(0)}|y|\,\nu(x,dy)=0,\\
    &\lim_{\varepsilon\downto 0}\sup_{x\in B_R(0)}\int_{B_\varepsilon(0)}|y|^{2}\,\nu(x,dy)=0,\\
    &\lim_{n\to\infty}\sup_{x\in B_R(0)}\int_{\real^{d}}|y|^{2} \, \lVert\nu([x]_n,dy)-\nu(x,dy)\rVert_{\mathrm{TV}}=0.
    \end{aligned}
\end{gather}
\end{proposition}

In the following proposition we discuss tightness conditions \eqref{t4}-\eqref{t6}.

\begin{proposition}\label{p2.5}
Assume that \eqref{t1.s} holds. The discretisation defined in \eqref{eq2.1} satisfies the conditions \eqref{t4}--\eqref{t6} if the following conditions hold:
\begin{align}
    \label{t2.s}\tag{\textbf{T2.S}}
    &\parbox[t]{.85\linewidth}{$\displaystyle{\lim_{r\to\infty}\sup_{x\in\real^{d}}\nu(x,B^{c}_r(0))=0}$;}\\
    \label{t3.s}\tag{\textbf{T3.S}}
    &\parbox[t]{.85\linewidth}{either there exists some $\rho>0$ such that $\nu(x,dy)$ is symmetric on $B_\rho(0)$ for all $x\in\real^{d}$, or there exists
    some $\rho>0$  such that for all $i=1,\dots,d$}\\
    &\notag\qquad\begin{gathered}
        \limsup_{\varepsilon\downto 0}\sup_{x\in\real^{d}}\bigg|\int_{B_\rho(0)\setminus B_{\varepsilon^{p}}(0)}y_i\,\nu(x,dy)\bigg|<\infty,\\
        \limsup_{\varepsilon\downto 0}\varepsilon\sup_{x\in\real^{d}}\nu\left(x,B_\rho(0)\setminus B_{\varepsilon^{p}}(0)\right)<\infty;
    \end{gathered}\\
    \label{t4.s}\tag{\textbf{T4.S}}
    &\parbox[t]{.85\linewidth}{there exists some \textup{(}alternatively: for all\textup{)} $\rho>0$ such that}\\
    &\notag\qquad\qquad
    \begin{gathered}
        \lim_{\varepsilon\downto 0}\sup_{x\in\real^{d}}\int_{B_\rho(0)\setminus B_\varepsilon(0)}|y|^{2}\,\nu(x,dy)<\infty.
    \end{gathered}
\end{align}
\end{proposition}

In the remaining part of this section we discuss some examples satisfying \eqref{c1.s}--\eqref{c6.s}.

\begin{example}[L\'evy processes]\label{e2.6}
    Let $\process{X}$ be a L\'evy process with semimartingale characteristics (L\'evy triplet) $(b,0,\nu(dy))$ with respect to some truncation function $h(x)$. Then, the conditions in \eqref{c1.s}--\eqref{c3.s} and, for the discretization \eqref{eq2.1} with $p<1/2$, and  \eqref{c4.s1}--\eqref{c6.s1}  are trivially satisfied.
\end{example}

\begin{example}[Stable-like processes]\label{e2.7}
    Let $\alpha:\real^{d}\to(0,2)$  be  a continuously differentiable function with bounded derivatives, such that $0<\underline{\alpha}=\alpha(x)=\overline{\alpha}<2$ for all $x\in\real^{d}$. Under these assumptions, it has been shown in \cite{Bass-1988,Hoh-2000, Kuhn-thesis, Negoro-1994}, \cite[Theorem 3.5]{Schilling-PTRF-1998} and \cite[Theorem 3.3]{Schilling-Wang-2013} that the operator \eqref{eq1.5} defines  a \emph{stable-like process}, i.e.\ a unique Feller process $\process{X}$ which is also a semimartingale. Its  (modified) characteristics (with respect to a symmetric truncation function $h(x)$) are of the form
    \begin{align*}
        B(h)_t &=0,\\
        A_t&=0,\\
        \tilde{A}^{i,k}_t &=\int_0^{t}\int_{\real^{d}}h_i(y)h_k(y)\frac{dy}{|y|^{d+\alpha(X_s)}}\,ds,\\
        N(ds,dy)&=\frac{dy\,ds}{|y|^{d+\alpha(X_s)}}.
    \end{align*}
    If $\alpha(x)\equiv\alpha\in (0,2)$ is constant, then we have a rotationally invariant $\alpha$-stable L\'evy process. Clearly, $\process{X}$ satisfies \eqref{c1.s}--\eqref{c3.s}. Using the discretisation \eqref{eq2.1} with $0<p\leq 1$, the continuity of $\alpha(x)$ and the dominated convergence theorem ensure that \eqref{c4.s1}--\eqref{c6.s1} hold, too.
\end{example}

\begin{example}[L\'evy-driven SDEs]\label{e2.8}
    Let $\process{L}$ be an $n$-dimensional L\'evy process and $\Phi:\real^{d}\to\real^{d\times n}$ be bounded and locally Lipschitz continuous. The SDE
    $$
        dX_t =\Phi(X_{t-})dL_t,\quad X_0=x\in\real^{d},
    $$
    admits a unique strong solution which is a Feller semimartingale, see  \cite{Kuhn-Preprint-2016}
    or  \cite[Theorem 3.5]{Schilling-Schnurr-2010}. In particular, if
    \begin{itemize}
    \item [(i)]
        $L_t=(l_t,t)$, $t\geq0$, where $\process{l}$ is a $d$-dimensional L\'evy process determined by the L\'evy triplet $(0,0,\nu(dy))$ such that the L\'evy measure $\nu(dy)$ is symmetric;
    \item [(ii)]
        $\Phi(x)=(\phi(x)\mathbb{I},\mathbb{0})$, $x\in\real^{d}$, where $\phi:\real^{d}\to\real$  is locally Lipschitz continuous and $0<\inf_{x\in\real^{d}}|\phi(x)|\leq\sup_{x\in\real^{d}}|\phi(x)|<\infty$, and $\mathbb{I}$ and $\mathbb{0}$ are the $d\times d$ identity matrix and the  $d$-dimensional  column vector of zeros,
    \end{itemize}
    then, according to \cite[Theorem 3.5]{Schilling-PTRF-1998},  $\process{X}$ is a $d$-dimensional Feller semimartingale which is determined by the (modified) characteristics (with respect to a symmetric truncation function $h(x)$) of the form
    \begin{align*}
        B(h)_t &=0,\\
        A_t &=0,\\
        \tilde{A}^{i,k}_t &=\int_0^{t}\int_{\real^{d}}h_i(y)h_k(y)\,\nu\left(dy/|\phi(X_s)|\right)\,ds,\\
        N(ds,dy) &=\nu\big(dy/|\phi(X_s)|\big).
    \end{align*}
    It is obvious, that the solution $\process{X}$ satisfies \eqref{c1.s}--\eqref{c3.s}. For the discretization \eqref{eq2.1}, the boundedness (away from zero and infinity) of $\phi(x)$, allows to prove the first two and the fourth condition in \eqref{c4.s1} and \eqref{c5.s1} as well as the first two conditions in \eqref{c6.s1}. Hence, it remains to prove that
    \begin{gather*}
        \lim_{n\to\infty}\sup_{x\in B_R(0)}
        \lVert\nu\left(\bullet/|\phi([x]_n)|\right)-\nu\left(\bullet/|\phi(x)|\right)\rVert_{\mathrm{TV}(B^{c}_\varepsilon(0))}=0,
        \quad \varepsilon>0,\\
    \lim_{n\to\infty}\sup_{x\in B_R(0)}\int_{B_1(0)}|y|^{2}
    \,\lVert\nu\left(dy/|\phi([x]_n)|\right)-\nu\left(dy/|\phi(x)|\right)\rVert_{\mathrm{TV}}=0,
    \intertext{and}
    \lim_{\varepsilon\downto 0}\varepsilon\sup_{x\in B_R(0)}\int_{B_1(0)\setminus B_{\varepsilon^{p}}(0)} |y|\,\nu\left(dy/|\phi(x)|\right)=0,
    \end{gather*}
    hold for all $R>0$.

    To prove the first relation we proceed as follows. Fix $R>0$ and $\varepsilon>0$. Due to the continuity and boundedness (away from zero and infinity) of $\phi(x)$ we can use dominated convergence theorem to see that the function
    $$
        x\mapsto\nu\left(B^{c}_\varepsilon(0)/|\phi(x)|\right)
    $$
    is (uniformly) continuous on any ball $B_R(0)$; this implies the desired relation.

    Essentially the same argument implies that for every $R>0$, the function
    $$
        x\mapsto\int_{B_1(0)}|y|^{2}\,\nu\left(dy/|\phi(x)|\right)
    $$
    is (uniformly) continuous on $B_R(0)$.

    Finally, we check the last relation. Fix $R>0$ and $0<p<1$. Then we have
    \begin{align*}
    \lim_{\varepsilon\downto 0}  \varepsilon & \sup_{x\in B_R(0)}\int_{B_1(0)\setminus B_{\varepsilon^{p}}(0)}|y|\,\nu\left(dy/|\phi(x)|\right)\\
    &=\lim_{\varepsilon\downto 0}\varepsilon^{1-p}\sup_{x\in B_R(0)}
    \int_{B_1(0)\setminus B_{\varepsilon^{p}}(0)}\varepsilon^{p}|y| \, \nu\left(dy/|\phi(x)|\right)\\
    &\leq\lim_{\varepsilon\downto 0}\varepsilon^{1-p}\sup_{x\in B_R(0)}
    \int_{B_1(0)\setminus B_{\varepsilon^{p}}(0)} |y|^{2}\,\nu\left(dy/|\phi(x)|\right)\\
    &\leq\lim_{\varepsilon\downto 0}\varepsilon^{1-p}\sup_{x\in B_R(0)}
    \int_{B_1(0)}|y|^{2}\,\nu\left(dy/|\phi(x)|\right)=0.
    \end{align*}

    Therefore, $\process{X}$ satisfies \eqref{c1.s}--\eqref{c6.s} if we use the discretisation \eqref{eq2.1} with $0<p<1$.
\end{example}

\begin{ack}
    Financial support through the Croatian Science Foundation (under Project 3526) (for Ante Mimica), the Croatian Science Foundation (under Project 3526)  and  the Dresden Fellowship Programme (for Nikola Sandri\'c) 
    is gratefully acknowledged.
\end{ack}
\bibliographystyle{alpha}
\bibliography{macha-ref}

\def\cprime{$'$}
\begin{thebibliography}{K{\"u}h16b}

\bibitem[Bas88]{Bass-1988}
R.~F. Bass.
\newblock Uniqueness in law for pure jump {M}arkov processes.
\newblock {\em Probab. Theory Related Fields}, 79(2):271--287, 1988.

\bibitem[BC08]{Burdzy-Chen-2008}
K.~Burdzy and Z.-Q. Chen.
\newblock Discrete approximations to reflected {B}rownian motion.
\newblock {\em Ann. Probab.}, 36(2):698--727, 2008.

\bibitem[BK08]{Bass-Kumagai-2008}
R.~F. Bass and T.~Kumagai.
\newblock Symmetric {M}arkov chains on {$\Bbb Z\sp d$} with unbounded range.
\newblock {\em Trans. Amer. Math. Soc.}, 360(4):2041--2075, 2008.

\bibitem[BKK10]{Bass-Kassmann-Kumagai-2010}
R.~F. Bass, M.~Kassmann, and T.~Kumagai.
\newblock Symmetric jump processes: localization, heat kernels and convergence.
\newblock {\em Ann. Inst. Henri Poincar\'e Probab. Stat.}, 46(1):59--71, 2010.

\bibitem[BKU10]{Bass-Kumagai-Uemura-2010}
R.~F. Bass, T.~Kumagai, and T.~Uemura.
\newblock Convergence of symmetric {M}arkov chains on {$\Bbb Z\sp d$}.
\newblock {\em Probab. Theory Related Fields}, 148(1-2):107--140, 2010.

\bibitem[BSW13]{Bottcher-Schilling-Wang-2013}
B.~B{\"o}ttcher, R.~L. Schilling, and J.~Wang.
\newblock {\em L\'evy matters. {III}}.
\newblock Springer, Cham, 2013.

\bibitem[CKK13]{Chen-Kim-Kumagai-2013}
Z.-Q. Chen, P.~Kim, and T.~Kumagai.
\newblock Discrete approximation of symmetric jump processes on metric measure
  spaces.
\newblock {\em Probab. Theory Related Fields}, 155(3-4):703--749, 2013.

\bibitem[DK13]{Deuschel-Kumagai-2013}
J-D. Deuschel and T.~Kumagai.
\newblock Markov chain approximations to nonsymmetric diffusions with bounded
  coefficients.
\newblock {\em Comm. Pure Appl. Math.}, 66(6):821--866, 2013.

\bibitem[EK86]{Ethier-Kurtz-Book-1986}
S.~N. Ethier and T.~G. Kurtz.
\newblock {\em Markov processes}.
\newblock John Wiley \& Sons Inc., New York, 1986.

\bibitem[Fol84]{Folland-Book-1984}
G.~B. Folland.
\newblock {\em Real analysis}.
\newblock John Wiley \& Sons, Inc., New York, 1984.

\bibitem[FOT11]{Fukushima-Oshima-Takeda-Book-2011}
M.~Fukushima, Y.~Oshima, and M.~Takeda.
\newblock {\em Dirichlet forms and symmetric {M}arkov processes}.
\newblock Walter de Gruyter \& Co., Berlin, 2011.

\bibitem[FU12]{Fukushima-Uemura-2012}
M.~Fukushima and T.~Uemura.
\newblock Jump-type {H}unt processes generated by lower bounded
  semi-{D}irichlet forms.
\newblock {\em Ann. Probab.}, 40(2):858--889, 2012.

\bibitem[Hin98]{Hino-1998}
M.~Hino.
\newblock Convergence of non-symmetric forms.
\newblock {\em J. Math. Kyoto Univ.}, 38(2):329--341, 1998.

\bibitem[HK07]{Husseini-Kassmann-2007}
R.~Husseini and M.~Kassmann.
\newblock Markov chain approximations for symmetric jump processes.
\newblock {\em Potential Anal.}, 27(4):353--380, 2007.

\bibitem[Hoh00]{Hoh-2000}
W.~Hoh.
\newblock Pseudo differential operators with negative definite symbols of
  variable order.
\newblock {\em Rev. Mat. Iberoamericana}, 16(2):219--241, 2000.

\bibitem[Jac05]{Jacob-Book-III-2005}
N.~Jacob.
\newblock {\em Pseudo differential operators and {M}arkov processes. {V}ol.
  {III}}.
\newblock Imperial College Press, London, 2005.

\bibitem[JS03]{Jacod-Shiryaev-2003}
J.~Jacod and A.~N. Shiryaev.
\newblock {\em Limit theorems for stochastic processes}.
\newblock Springer-Verlag, Berlin, second edition, 2003.

\bibitem[Kim06]{Kim-2006}
P.~Kim.
\newblock Weak convergence of censored and reflected stable processes.
\newblock {\em Stochastic Process. Appl.}, 116(12):1792--1814, 2006.

\bibitem[KS03]{Kuwae-Shioya-2003}
K.~Kuwae and T.~Shioya.
\newblock Convergence of spectral structures: a functional analytic theory and
  its applications to spectral geometry.
\newblock {\em Comm. Anal. Geom.}, 11(4):599--673, 2003.

\bibitem[K{\"u}h16a]{Kuhn-thesis}
F.~K{\"u}hn.
\newblock {\em {P}robability and {H}eat {K}ernel {E}stimates for
  {L}\'evy(-{T}ype) {P}rocesses}.
\newblock PhD thesis, der Fakult{\"a}t Mathematik und Naturwissenschaften der
  Technischen Universit{\"a}t Dresden, 2016.

\bibitem[K{\"u}h16b]{Kuhn-Preprint-2016}
F.~K{\"u}hn.
\newblock Solutions of {L}\'evy-driven {S}{D}{E}s with unbounded coefficients
  as {F}eller processes.
\newblock {\em Preprint. Available at arXiv:1610.02286}, 2016.

\bibitem[Mos94]{Mosco-1994}
U.~Mosco.
\newblock Composite media and asymptotic {D}irichlet forms.
\newblock {\em J. Funct. Anal.}, 123(2):368--421, 1994.

\bibitem[MR92]{Ma-Rockner-Book-1992}
Z.~M. Ma and M.~R{\"o}ckner.
\newblock {\em Introduction to the theory of (nonsymmetric) {D}irichlet forms}.
\newblock Springer-Verlag, Berlin, 1992.

\bibitem[Neg94]{Negoro-1994}
A.~Negoro.
\newblock Stable-like processes: construction of the transition density and the
  behavior of sample paths near {$t=0$}.
\newblock {\em Osaka J. Math.}, 31(1):189--214, 1994.

\bibitem[Nor98]{Norris-Book-1998}
J.~R. Norris.
\newblock {\em Markov chains}.
\newblock Cambridge University Press, Cambridge, 1998.

\bibitem[Sch98]{Schilling-PTRF-1998}
R.~L. Schilling.
\newblock Growth and {H}\"older conditions for the sample paths of {F}eller
  processes.
\newblock {\em Probab. Theory Related Fields}, 112(4):565--611, 1998.

\bibitem[SS10]{Schilling-Schnurr-2010}
R.~L. Schilling and A.~Schnurr.
\newblock The symbol associated with the solution of a stochastic differential
  equation.
\newblock {\em Electron. J. Probab.}, 15:1369--1393, 2010.

\bibitem[SU12]{Schilling-Uemura-2012}
R.~L. Schilling and T.~Uemura.
\newblock On the structure of the domain of a symmetric jump-type {D}irichlet
  form.
\newblock {\em Publ. Res. Inst. Math. Sci.}, 48(1):1--20, 2012.

\bibitem[SV06]{Stroock-Varadhan-Book-2006}
D.~W. Stroock and S.~R.~S. Varadhan.
\newblock {\em Multidimensional diffusion processes}.
\newblock Springer-Verlag, Berlin, 2006.
\newblock Reprint of the 1997 edition.

\bibitem[SW13]{Schilling-Wang-2013}
R.~L. Schilling and J.~Wang.
\newblock Some theorems on {F}eller processes: transience, local times and
  ultracontractivity.
\newblock {\em Trans. Amer. Math. Soc.}, 365(6):3255--3268, 2013.

\bibitem[SW15]{Schilling-Wang-2015}
R.~L. Schilling and J.~Wang.
\newblock Lower bounded semi-{D}irichlet forms associated with {L}\'evy type
  operators.
\newblock {\em In: Z-Q. Chen, N. Jacob, M. Takeda, T. Uemura: {F}estschrift
  {M}asatoshi {F}ukushima. {I}n {H}onor of {M}asatoshi {F}ukushima's {S}anju.
  World Scientific, New Jersey}, pages 507--527, 2015.

\bibitem[SZ97]{Stroock-Zheng-1997}
D.~W. Stroock and W.~Zheng.
\newblock Markov chain approximations to symmetric diffusions.
\newblock {\em Ann. Inst. H. Poincar\'e Probab. Statist.}, 33(5):619--649,
  1997.

\bibitem[T\"06]{Tolle-Thesis-2006}
J.~M. T\"olle.
\newblock Convergence of non-symmetric forms with changing reference measures.
\newblock Master's thesis, Faculty of Mathematics, University of Bielefeld,
  2006.

\end{thebibliography}

\end{document}